\documentclass[12pt]{amsart}
\usepackage{amsmath,amssymb,amsfonts,amscd}
\usepackage{graphicx}
\usepackage{xypic}
\usepackage[margin=1.14in]{geometry}
\usepackage{stmaryrd}

\theoremstyle{plain}
\newtheorem{thm}{Theorem}
\newtheorem{lem}[thm]{Lemma}
\newtheorem{cor}[thm]{Corollary}
\newtheorem{prop}[thm]{Proposition}
\newtheorem{remark}[thm]{Remark}
\newtheorem{defn}[thm]{Definition}
\newtheorem{defn-prop}[thm]{Proposition-Definition}
\newtheorem{ex}[thm]{Example}

\newtheorem{assum}[thm]{Assumption}

\usepackage{color}

\definecolor{lightgrey}{rgb}{0.8, 0.84, 0.8}

\def\la{\langle}
\def\ra{\rangle}
\def\bal{\begin{aligned}}
\def\eal{\end{aligned}}

\newcommand{\eq}[2]{\begin{equation}\label{#1}#2 \end{equation}}
\newcommand{\ml}[2]{\begin{multline}\label{#1}#2 \end{multline}}

\newcommand{\surj}{\twoheadrightarrow}
\newcommand{\inj}{\hookrightarrow}

\newcommand{\rank}{{\rm rank}}

\newcommand{\sD}{{\mathcal D}}
\newcommand{\sE}{{\mathcal E}}

\newcommand{\sG}{{\mathcal G}}
\newcommand{\sH}{{\mathcal H}}

\newcommand{\sL}{{\mathcal L}}
\newcommand{\sM}{{\mathcal M}}

\newcommand{\sO}{{\mathcal O}}

\newcommand{\sR}{{\mathcal R}}

\newcommand{\sW}{{\mathcal W}}
\newcommand{\sX}{{\mathcal X}}

\newcommand{\C}{{\mathbb C}}

\newcommand{\E}{{\mathbb E}}
\newcommand{\F}{{\mathbb F}}
\newcommand{\G}{{\mathbb G}}
\renewcommand{\H}{{\mathbb H}}

\renewcommand{\P}{{\mathbb P}}
\newcommand{\Q}{{\mathbb Q}}
\newcommand{\R}{{\mathbb R}}

\newcommand{\Z}{{\mathbb Z}}

\newcommand{\ve}{{\varepsilon}}

\title{Gamma functions, monodromy and Frobenius constants}
\author{Spencer Bloch, Masha Vlasenko}

\address{5765 S. Blackstone Ave., Chicago, IL 60637, USA}
\email{spencer\_bloch@yahoo.com}

\address{Institute of Mathematics of the Polish Academy of Sciences, Sniadeckich 8, 00-656 Warsaw, Poland}
\email{masha.vlasenko@gmail.com}

\thanks{
Work of the second author was supported by the National Science Centre of Poland (NCN), grant UMO-2016/21/B/ST1/03084.}

\begin{document}
\maketitle

\section*{Introduction}\label{intro}
In an important paper \cite{GZ}, Golyshev and Zagier introduce what we will refer to as {\it Frobenius constants} $\kappa_{\rho, n}$  associated to an ordinary linear differential operator $L$ with a reflection type singularity at $t=c$. For every other regular singularity $t=c'$ and a homotopy class of paths $\gamma$ joining $c'$ and $c$, constants $\kappa_{\rho,n}=\kappa_{\rho,n}(\gamma)$ describe the variation around $c$ of the Frobenius solutions $\phi_{\rho,n}(t)$ to $L$ defined near $t=c'$ and continued analytically along $\gamma$.  (Here $\rho \in \C$ are local exponents of $L$ at $t=c'$, see Definition~\ref{full-apery-defn} below.)   Golyshev and Zagier show in certain cases that the $\kappa_{\rho,n}$ are periods, and they raise the question quite generally how to describe the $\kappa_{\rho,n}$ motivically.

The purpose of this work is to develop the theory (first suggested to us by Golyshev) of {\it motivic Mellin transforms} or {\it motivic gamma functions}. Our main result (Theorem \ref{gamma-apery-thm}) relates the generating series $\sum_{n = 0}^\infty \kappa_{\rho,n} (s-\rho)^n$ to the Taylor expansion at $s=\rho$ of a generalized gamma function, which is a Mellin transform of a solution of the dual differential operator $L^\vee$. It follows from this that the numbers $\kappa_{\rho,n}$ are always periods when $L$ is a geometric differential operator (Corollary~\ref{kappas-are-periods}). 

Briefly, the content of various sections of the paper is as follows. Section \ref{sec:motivic-gamma} describes our basic approach to periods associated to local systems on an open curve. It will be technically convenient to interpret homology of local systems in terms of group homology of the fundamental group. Mellin transforms and $\Gamma$-functions are described in these terms. The difference equation satisfied by a $\Gamma$-function is proven. Section \ref{sec:gamma-monodromy} focuses on the semi-local picture as in \cite{GZ}. Our local system has a regular singularity at $0$ and a reflection type singularity at some point $c$. In contrast to op. cit., here $c$ is not necessarily the closest singularity to $0$. We fix a path from $0$ to $c$ not passing through any other singularities, and take an open connected set $V$ containing the path with punctures at $0, c$ and fundamental group $F_2$. We study $\Gamma$-functions associated to paths on $V$  and find that, up to simple ambiguities, such a function is unique. Section~\ref{sec:Apery} contains a quick sketch of the Frobenius method including the inhomogeneous solutions (\emph{higher Frobenius functions, op. cit.}) We then compute Frobenius constants of hypergeometric connections and mention a few other motivating numerical examples.  The main theorem is stated, which relates the unique gamma function found in Section~\ref{sec:gamma-monodromy} to the monodromy of Frobenius solutions. In Section~\ref{sec:Frob-def-and-gamma} we prove the main theorem. Finally, in Section~\ref{sec:Frob-near-MUM} we relate Frobenius constants of geometric connections to periods of limiting mixed Hodge structures. 

\section{Motivic $\Gamma$-functions}\label{sec:motivic-gamma}
Let $C$ be a complete, smooth algebraic curve over $\C$, and let $S\subset C$ be a non-empty, finite set of points. Let $M$ be an algebraic connection on $U:=C-S$. The de Rham cohomology of the connection, $H_{DR}^*(U,M)$ is the cohomology of the $2$-term complex of modules (placed in degrees $0, 1$) over $\Gamma(U, \sO_U)$
\[
M \xrightarrow{\nabla_M} M\otimes \Omega^1_U.
\]

Recall by definition a {\it solution} for $M$ is a horizontal section of the dual connection $\mu \in M^{\vee,\nabla^\vee=0}_{an}$. Whereas $M$ and $H^*_{DR}(U,M)$ are purely algebraic in nature, interesting horizontal sections are usually multi-valued and only defined locally analytically, so we consider the analytic connection $M^\vee_{an}$ on $U_{an}$. By coupling solutions to suitable topological chains in $C_{an}$, one defines rapid decay homology groups $H_{*,rd}(U_{an},M^\vee_{an})$, \cite{BE}, and there is a period pairing which is perfect pairing of finite dimensional vector spaces
\[H^1_{DR}(U,M)\times H_{1,rd}(U,M^\vee) \to \C.
\]
This construction is valid even when $M$ has irregular singular points. It can be used, for example, to construct the classical Bessel and confluent hypergeometric differential equations. In this paper, we will consider only the case where $M$ has regular singular points. In this case, one can ignore the rapid decay condition for homology and work with the standard topological homology of the local system of solutions. It is, however, worth remarking that the key point in the proof of the main result, theorem \ref{gamma-apery-thm}, is the adjunction property of the bracket, formula (\ref{int-by-parts}). This adjunction property does not require regular singular points. We expect that much of the theory developed here will generalize to cover Mellin transforms of solutions of (not necessarily regular) differential equations. 

 We write $\sM^\vee:=M^{\vee,\nabla^\vee=0}_{an}$ for the local system. Note that this local system can often be defined over a subfield $K\subset \C$. 
Homology can be computed over $K$, e.g., by fixing a basepoint $p\in U_{an}$ and interpreting $\sM^\vee$ as a representation of $\pi_1(U_{an},p)$ on $\sM^\vee_p$. Let $\widetilde{U}_{an} \to U_{an}$ be the universal cover, and let $C_*(\widetilde{U}_{an},K)$ be the complex of topological chains on the universal cover. Homology is then defined by coupling the chains to the representation
\[
H_*(U_{an},\sM^\vee):=H_*(C_*(\widetilde{U}_{an},K)\otimes_{K[\pi_1(U,p)]} \sM^\vee_p).
\]
Concretely, in degree $1$, the period pairing can be represented as follows. For us, $\Omega^1_U$ will always be a free, rank $1$ module. We fix $\omega\in \Omega^1_U$ a generator. A de Rham $1$-cocycle $c$ lifts to an element $m\otimes \omega \in M\otimes \Omega^1_U$. An homology class $\mu\in H_1(U,\sM^\vee)$ (to simplify notation, we no longer write the subscript ${}_{an}$) can be represented by a finite sum $\sum_j \sigma_j \otimes \ve_j $ where $\sigma_j \in \pi_1(U,p)$, $\ve_j \in \sM_p^\vee$ and $\sum_j \sigma_j \ve_j = \sum_j \ve_j$. The latter condition means that $\mu$ is a $1$-cycle (and not just a $1$-chain). The resulting period is
\[
\langle c,\mu\rangle= \sum_j \int_{\sigma_j} \langle m,\ve_j\rangle \omega.
\]

\begin{ex}\label{Picard-Fuchs-ex}
Let $f: X \to U$ be a smooth, proper map of algebraic varieties. Let $M:=H^n_{DR}(X/U)$ be the relative de Rham cohomology, endowed with the algebraic Gau\ss--Manin connection $\nabla$. Here we are totally in the realm of algebraic geometry, so if, for example, $f, X, U$ are all defined over a subfield $k\subset \C$, then our connection $M$ will be defined over $k$ as well. In the Gau\ss--Manin setup, solutions typically arise from continuously varying closed chains on the fibres. Since the homology of the fibres is a local system defined over $\Q$, we can think of $\sM^\vee_{an}$ as having a Betti structure and take $K=\Q$. 
\end{ex}
\begin{ex}\label{tensor-periods}The category of connections on $U$ has a tensor product, so we can add interest to our study by coupling e.g. a Gau\ss--Manin connection $M$ as in the previous example to one of a number of standard connections on $U$. The effect of tensoring connections is to multiply solutions appearing in the period integral. 

To avoid some technicalities, let us consider the case $U \subseteq \G_m$. We use $t$ as a coordinate on $\G_m$. Three examples are \newline\noindent
(i)(Mellin transform) Take the connection on $\sO_U$ given by $\nabla_{Mellin}(1):=s \, dt/t$. Somewhat abusively, this connection is denoted $t^s$. It has $t^{s}$ as solution. The period integrals for $M \otimes t^s$ are of the shape $\int_\sigma \langle m,\ve \rangle t^s\omega$. Our periods become functions of $s$.
\newline\noindent
(ii)(Fourier transform) Define a connection on $\sO_U$ by $\nabla_{Fourier}(1) := s \, dt$.  The solution is $e^{st}$, and period integrals for $M \otimes e^{st}$ are of the shape $\int_\sigma \langle m,\ve \rangle e^{st}\omega$. Here again $s$ is a parameter.
\newline\noindent
(iii)(Kummer connection) Let $K_t$ be the rank $2$ connection with solutions $\ve_1, \ve_2$ such that there exists $m\in M$ with $\langle m,\ve_1\rangle = 1$ and $\langle m,\ve_2\rangle = \log t$.  As an exercise, the reader can write out the connections $Sym^n(K_t)$ and describe the integrands involved in calculating periods for $M\otimes Sym^n(K_t)$. 
\end{ex}

The examples just given can be also considered in the case when $U$ is an open curve and $t \in \Gamma(U,\sO_U^\times)$ is an invertible function on $U$.  We can now give a vague definition of the object of interest:   

\bigskip

\noindent{\bf Definition~\ref{gamma-def-0}'.} \emph{A  gamma function is the function of $s$ associated to a period of the Mellin transform of a  connection $M$ on $U$. If $M$ is a Gau\ss--Manin connection, then the resulting gamma functions are called \emph{motivic}.  }

\bigskip

In this paper we shall work with connections $M$ on an open subset $U \subset \G_m$. We use $t$ as a coordinate on $\G_m$. 

Let us associate explicit gamma functions with homology classes in $H_1(U,\sM^\vee \otimes t^s)$. For that we fix a basepoint $t=p$ and consider the representation of $\pi_1(U,p)$ on the stalk 
\eq{Mellin-stalk}{
(\sM^\vee \otimes t^s)_p \cong \sM_p^\vee \otimes_K K[e^{\pm 2 \pi i s}],
}
where the homotopy group acts on the second component of the tensor product through the monodromy of $t^s$. 

\bigskip
\begin{defn}\label{gamma-def-0} \emph{Fix $m \otimes \frac{dt}{t} \in M \otimes \Omega^1_U$. A homology class $\xi \in H_1(U,\sM^\vee \otimes t^s)$ can be represented by a 1-cycle
\[
\xi \sim \sum _j \sigma_j \otimes \ve_j \otimes e^{2 \pi i s n_j},
\]
where the sum is finite, $\sigma_j$ are loops based at $p$, $\ve_j \in \sM^\vee_p$ are solutions in a neighbourhood of~$p$ and $n_j \in \Z$. The respective gamma function is given by
\eq{gamma-def}{
\Gamma_\xi(s) =  \sum_j e^{2 \pi i s n_j} \int_{\sigma_j} \langle m , \ve_j \rangle t^{s-1} dt.
}
Here we also assume that a branch of $t^s$ at the base point $p$ is fixed. It is thus the same branch in every integral in the right-hand sum, while the coefficient $e^{2 \pi i s n_j}$ accounts for the possibility of choosing different branches. 
}
\end{defn}

Note that function $\Gamma_\xi(s)$ given by~\eqref{gamma-def} is an entire function of~$s$. 

\begin{lem}\label{gamma-well-defined} The right-hand side of~\eqref{gamma-def} depends only on the homology class of $\xi$ in $H_1(U,\sM^\vee \otimes t^s)$.
\end{lem}

For the proof, it will be convenient to identify 
\[
H_1(U,\sM^\vee \otimes t^s) \cong H_1(\pi_1(U,p), \sM_p^\vee \otimes_K K[e^{\pm 2 \pi i s}] ).
\]
We will use the inhomogeneous bar complex $B_*[\pi_1]$ for the group $\pi_1=\pi_1(U,p)$ tensored over $\Z[\pi_1]$ with the representation $V=\sM_p^\vee \otimes_K K[e^{\pm 2 \pi i s}]$. The following formulas for the differentials in low degrees in this complex will arise in computations throughout the paper:
\eq{1-boundary}{
\partial ([g_1] \otimes v ) = \sigma_1 v - v
}
and
\eq{2-boundary}{
\partial ([g_1, g_2] \otimes v ) = [g_2] \otimes g_1 v - [g_1 g_2] \otimes v + [g_1] \otimes v
}
for any $g_1,g_2 \in \pi_1$ and $v \in V$. For 1-chains we will omit the bracket and write $[\sigma] \otimes v$ simply as $\sigma \otimes v$. 

\begin{proof}[Proof of Lemma~\ref{gamma-well-defined}.] According to~\eqref{2-boundary}, boundaries of 2-chains are generated over $K[e^{\pm 2 \pi i s}]$ by expressions of the form 
\[
\partial ([\sigma_1, \sigma_2] \otimes \ve \otimes 1) = \sigma_2 \otimes \sigma_1 (\ve \otimes 1) - \sigma_1 \sigma_2 \otimes \ve \otimes 1 + \sigma_1 \otimes \ve \otimes 1\,. 
\]
Vanishing of~\eqref{gamma-def} on such expressions is the composition formula. Namely, to integrate $\la m, \ve\ra t^{s-1} dt$ over $\sigma_1 \sigma_2$ we first integrate it over $\sigma_1$ and then integrate $\sigma_1(\la m, \ve\ra t^{s-1} dt)=\la m,\sigma_1(\ve \otimes 1)\ra t^{s-1} dt$ over $\sigma_2$. With the integer $n=n(\sigma_1) \in \Z$ such that $\sigma_1 t^s = e^{2 \pi i s \, n}t^s$, we can write this as
\[
\int_{\sigma_1 \sigma_2} \la m, \ve\ra t^{s-1} dt = \int_{\sigma_1} \la m, \ve\ra t^{s-1} dt  + e^{2 \pi i s n} \int_{\sigma_2} \la m, \sigma_1\ve\ra t^{s-1} dt\,.
\]
\end{proof}

Note that the stalk~\eqref{Mellin-stalk} is a free module over $K[e^{\pm 2 \pi i s}]$ of rank $\dim_K \sM_p^\vee = \rank(M)$. The action of $\pi_1(U,p)$ commutes with the $K[e^{\pm 2 \pi i s}]$-module structure and therefore $H_1(U, \sM^\vee \otimes t^s)$ is a $K[e^{\pm 2 \pi i s}]$-module. It is clear that evaluation
\[
\xi \mapsto \Gamma_{\xi}(s) 
\]
in~\eqref{gamma-def} is $K$-linear and commutes with multiplication by $e^{2 \pi i s}$. Therefore we obtain a $K[ e^{\pm 2 \pi i s}]$-module of gamma functions. As a module, this is the quotient of $H_1(U, \sM^\vee \otimes t^s)$ by classes for which the respective gamma functions vanish. It follows that all gamma functions are $K[e^{\pm 2 \pi i s}]$-linear combinations of a finite number of them:

\begin{prop}\label{gamma-f-g} The $K[e^{\pm 2 \pi i s}]$-module of gamma functions is finitely generated. 
\end{prop}
\begin{proof}  As was mentioned at the beginning of this section, homology of our local system can be computed using the chain complex of the universal cover $\tilde U$:   
\[
H_*(U,\sM^\vee \otimes t^s) \cong H_*\Bigl(C_*(\widetilde{U},K)\otimes_{K[\pi_1(U,p)]} (\sM^\vee \otimes t^s)_p \Bigr).
\]
As a topological space, $U$ is homotopic to a finite CW-complex, and hence the chain complex of the universal cover $C_*=C_*(\widetilde{U},K)$ is homotopic to a complex of finitely generated $K[\pi_1]$-modules. Since the stalk representation $V=\sM^\vee_p \otimes_K K[e^{\pm 2 \pi i s}]$ is a finitely generated module over a Noetherian ring $R=K[e^{\pm 2 \pi i s}]$, $C_* \otimes_{K[\pi_1]} V$ is homotopic to a complex of finitely generated $R$-modules, so it has finitely generated homology. In particular, $H_1(U, \sM^\vee \otimes t^s)$ is finitely generated.
\end{proof}

\begin{ex}\label{baby-example} The double cover $f: \P^1_y \to \P^1_t$ given by $t=1-y^2$ is ramified at $t=1,\infty$. Removing the point $t=0$ yields
\[
C^\circ := \P^1_y \setminus \{1,-1,0,\infty\} \overset{f^\circ} \to U := \P^1_t \setminus \{0,1,\infty\}\,. \\
\]
We have $f^\circ_*\sO_{C^\circ} = \sO_U \oplus \sO_U [y]$. The line bundle $M := \sO_U [y]$ carries a connection with $\nabla_{d/dt}[y]=-\frac1{2(1-t)}[y]$. Choose a point $p \in U$ and a horizontal section of the dual bundle $\ve \in \sM^\vee_p$. Let $\sigma_0$ and $\sigma_1$ be the loops around $0$ and $1$ on $\P^1_t$. In $\sM_p^\vee \otimes_\Q \Q[e^{\pm 2 \pi i s}]$ we have $\sigma_0 (\ve \otimes 1) =\ve \otimes e^{2 \pi i s}$ and $\sigma_1 (\ve \otimes 1) = - \ve \otimes 1$. The loop $\sigma=\sigma_1 \sigma_0\sigma_1\sigma_0^{-1}$ fixes $\ve \otimes 1$, so $\xi=\sigma_0^{-1}\sigma_1\sigma_0\sigma_1 \otimes (\ve\otimes 1)$ is a 1-cycle. The associated gamma function is 
\[
\Gamma_\xi(s) = \int_{\sigma}  \la [y], \ve \ra t^{s-1} dt\,.  
\]
The 1-cycle condition here converts into the fact that $\sigma$ is a closed path along which $t^{s} \la [y], \ve \ra $ is single-valued, thus the above integral is well defined. 

Since $\ve$ is horizontal, the pairing $\la [y], \ve \ra$ is a solution to the differential operator $(1-t) \frac{d}{dt} - \frac1{2}$, and hence it is a constant multiple of $(1-t)^{-1/2}$. Possibly rescaling $\ve$, the following beta integral is a motivic gamma function:
\[\bal
\Gamma_\xi(s) &= \int_{\sigma_1 \sigma_0\sigma_1\sigma_0^{-1}} t^{s-1} (1-t)^{-1/2} dt = \int_0^1 - \int_1^0 - e^{2 \pi i s} \int_0^1 + e^{2 \pi i s} \int_1^0 \\
&= 2 (1-e^{2 \pi i s}) \int_0^1 t^{s-1} (1-t)^{-1/2} dt = 2 (1-e^{2 \pi i s}) \frac{\Gamma(s)\Gamma(1/2)}{\Gamma(s+1/2)}\,.
\eal\]
Note that $\Gamma_\xi(s)$ is an entire function of $s$. This is a feature of all our generalized gamma functions, as one can easily see from their definition~\eqref{gamma-def}. 
\end{ex}
\begin{remark}The reader will have noticed that the classical gamma function $\Gamma(s)=\int_0^\infty t^{s-1} e^{-t} dt$ is not a motivic gamma function. The connection in this case is $\nabla(1)=dt$ which has an irregular singular point at $t=\infty$. The notion of period can be extended to the irregular case using some form of ``rapid decay'' homology ( \cite{BE}, \cite{FJ}). The classical path of integration from $0$ to $\infty$ is not allowed because the Mellin connection $\nabla_{Mellin}(1)=sdt/t$ has a singular point at $t=0$. However, if we replace $[0,\infty]$ with a ``keyhole'' path starting at $\infty$, following the positive real axis to $+\ve$, looping counterclockwise about $0$, and then going back to $+\infty$, the resulting ``period'', $(e^{2\pi is}-1)\Gamma(s)$ suggests a natural generalization of motivic gammas to the irregular case. Notice that again the period is an entire function of $s$. 
\end{remark}
\bigskip

We will say a function $f(s)$ satisfies a difference equation of length~$a$ if there exist polynomials $p_0(s), \ldots,p_a(s)$ such that
\[
0 = \sum_{j=0}^a p_j(-s-j)f(s+j).
\]
Note that if $f(s)$ satisfies a difference equation  of length~$a$ and $g(s)$ is a polynomial, then $g(s)f(s)$ also satisfies a difference equation  of length~$a$. Indeed, $\sum_{j=0}^a q_j(-s-j)g(s+j)f(s+j)=0$ with $q_j(s)=p_j(s) \prod_{0 \le k \le a, k \ne j} g(-s-j+k)$.

It is a general fact that Mellin transforms satisfy difference equations (see~\cite{LS}). Remember that our choice is to fix $m\otimes\frac{dt}{t} \in M \otimes \Omega_U^1$, in which case we obtain a finitely generated $K[e^{\pm 2 \pi i s}]$-module of gamma functions indexed by $\xi \in H_1(U,\sM^\vee \otimes t^s)$. All of them satisfy the same difference equation which can be found as follows.Let $r = \rank(M)$ and consider the derivation $D=t \, d/dt$. Then there exist $q_0,\ldots,q_r \in \sO_U$ such that the differential operator
\[
L = q_0(t) \, D^r + q_1(t) \, D^{r-1} + \ldots + q_r(t) 
\]
annihilates $m$. Here and throughout the paper we shall adopt the convention that $D$ acts on $M$ via $\nabla_M(D)$; for example, $Lm=0$ means that $\sum_{j=0}^r q_{r-j} \nabla_M(D)^j \, m = 0$.
Observe that for a solution $\ve \in \sM^\vee$ the analytic function $\phi = \langle m, \ve \rangle \in \sO^{an}$ satisfies the differential equation $L \phi = 0$.

\begin{prop}\label{functional_equation} Assume that $q_j=q_j(t) \in \C[t]$ for $0 \le j \le r$ and rearrange terms in the differential operator
\eq{q-to-p}{
L = \sum_{j=0}^r q_{r-j}(t) D^j = p_0(D) + t \, p_1(D) +  \ldots + t^a \, p_a(D)
}
with polynomials $p_0,\ldots,p_a \in \C[D]$ of degree at most $r$. Then for every homological class $\xi $ the respective gamma function $\Gamma_\xi(s)$ satisfies the difference equation
\eq{difference-eq}{
\sum_{j=0}^a p_j(-s-j)\Gamma_\xi(s+j) = 0\,.
}
\end{prop}
\begin{proof} For any $m' \in M$ let us denote 
\[
\Gamma_\xi(m',s) = \sum_j e^{2 \pi i s n_j} \int_{\sigma_j} \langle m', \ve_j \rangle t^s dt/t,
\]
which is just the gamma function~\eqref{gamma-def} corresponding to $m' \otimes dt/t \in M \otimes \Omega^1_U$. Since $\xi$ is a 1-cycle, the condition
\[
\partial \xi = \sum_j e^{2 \pi i s n_j}(\sigma_j-1) (\ve_j \otimes 1) = 0
\]
together with the fundamental theorem of calculus imply
$$ \sum_j e^{2 \pi i s n_j} \int_{\sigma_j} D(\langle m' , \ve_j \rangle t^{s}) dt/t = 0.
$$
Expanding out, using the fact that $\ve_j$ is a horizontal section of $M^\vee$, we get
$$-s \, \Gamma_\xi(m',s)=\sum_j e^{2 \pi i s n_j} \int_{\sigma_j} \langle  D m' , \ve_j \rangle t^{s} dt/t=\Gamma_\xi(D m',s).
$$
One also has trivially
$$\Gamma_\xi(t \, m',s) = \Gamma_\xi(m',s+1).
$$
Since $\sum_{j=0}^r q_j(t) D^j m=0$, using formula~\eqref{q-to-p} we get
\ml{}{0 = \Gamma_\xi(\sum_{j=0}^r q_j(t) D^j m,s) = \sum_{j=0}^a\Gamma_\xi(p_j(D)m,s+j) = \\
\sum_{j=0}^a p_j(-s-j)\Gamma_\xi(m,s+j).
}
\end{proof}

\section{Monodromy and existence of gamma functions}\label{sec:gamma-monodromy}

The setting is as in the first section: we are given an algebraic connection $M$ with regular singularities on an open curve $U \subset \G_m$; the coordinate on $\G_m$ is $t$. The local system of solutions is denoted by $\sM^\vee = (M_{an}^\vee)^{\nabla^\vee=0}$. This local system is defined over a field $K \subseteq \C$. 

An important class of motivic $\Gamma$-functions arises when $\xi$ lifts to a class in $H_1(V, \sM^\vee\otimes t^s)$, where $V\subset U$ is an open in the $\C$-analytic topology neighbourhood of a path joining $t=0$ to another singular point $c  \ne 0,\infty$ of $M$. Namely, we make the following

\begin{assum}\label{V-assum} We have $V=V_0\cup V_c$ where $V_0$ and $V_c$ are punctured disks centered at $0$ and $c$ respectively, and $V_0\cap V_c$ is contractible.

We fix a base point $p \in V_0 \cap V_c$ and denote by $\sigma_0$ and $\sigma_c$ the loops around $0$ and $c$ respectively.
\end{assum}

\noindent Under this assumption $\pi_1(V,p)$ is a free group on the generators $\sigma_0$ and $\sigma_c$. The image of the corestriction map 
\eq{Cor}{
Cor: \; H_1(V_c, \sM^\vee \otimes t^s) \to H_1(V, \sM^\vee\otimes t^s)
} 
is the $K[e^{\pm 2\pi i s}]$-submodule generated by classes of cycles $\xi = \sigma_c \otimes \ve \otimes 1$ where $\ve \in \sM_p^\vee$ is a $\sigma_c$-invariant solution. The respective gamma functions can be evaluated easily. Namely, since $c$ is a regular singularity, the analytic function $\langle m, \ve \rangle$ is meromorphic at $t=c$. Expanding it in the Laurent series  $\langle m, \ve \rangle = a_{-k}(t-c)^{-k} + a_{1-k}(t-c)^{1-k} + \ldots$ we find that
\[
\Gamma_{\xi}(s) = \int_{\sigma_c} t^{s-1} \langle m, \ve \rangle dt = \sum_{j=1}^k a_{-j} \binom{s-1}{j-1} c^{s-j},
\]
which is simply $c^s$ multiplied by a polynomial in $s$. Note that this gamma function is identically zero if $\langle m, \ve \rangle$ is holomorphic at $t=c$. Thus one can always multiply $m$ by a power of $(t-c)$ to ensure that the gamma functions corresponding to elements in the image of~\eqref{Cor} vanish. 

\begin{lem}\label{gamma-generators-prop} Let $d = \dim_K Image(\sigma_c-1|\sM^\vee_p)$ be the rank of the variation of the local monodromy of $\sM^\vee$ around $t=c$. The cokernel of~\eqref{Cor} is a free $K[e^{\pm 2 \pi i s}]$-module of rank $d$.
\end{lem}

In the course of proof of this lemma the following fact will be used: 

\begin{lem}\label{free-H-1-lemma} Write $\Z \cong u^\Z$ for the free abelian group on one generator, written multiplicatively. Let $A$ be an abelian group, and suppose we are given $\phi: A \to A$ an automorphism. We view $A$ as a $u^\Z$-module, with $u$ acting as $\phi$. Then 
$$H_1(u^\Z, A) \cong A^{\phi=id}. 
$$
\end{lem}

\begin{proof}We compute in the bar complex using~\eqref{1-boundary} and~\eqref{2-boundary}: 
\[\bal
&\partial (u^j \otimes a) = u^ja-a, \\
&\partial([u^j,u^k]\otimes a) = u^k\otimes u^ja - u^{j+k}\otimes a + u^j\otimes a.
\eal\]
Clearly, any $1$-chain $\sum_i g_i\otimes a_i$ is equivalent modulo boundaries of $2$-chains as above to a $1$-chain of the form $1 \otimes a_0 + u \otimes a_1 + u^{-1} \otimes a_{-1}$. Similarly, taking $j=k=0$ shows that $1$-chains $1\otimes a$ are boundaries, and $u^{-1}\otimes a \sim -u\otimes ua$. In this way, any $1$-chain is equivalent to a $1$-chain of the form $u \otimes a$.  Since 
$$\partial( u\otimes a)  = \phi(a) - a,
$$
this latter 1-chain is a 1-cycle if and only if $a \in A^{\phi=id}$. 
\end{proof}

\begin{proof}[Proof of Lemma~\ref{gamma-generators-prop}.]
Consider the long exact sequence for the relative homology with coefficients in $\sL = \sM^\vee \otimes t^s$:
\eq{rel-hom-les}{
\ldots \to H_1(V_c;\sL) \overset{Cor}{\to} H_1(V;\sL) \to H_1(V,V_c;\sL) \overset{b}\to H_0(V_c;\sL) \to \ldots
}
From exactness, the cokernel of $Cor$ is isomorphic to the kernel of the connecting map $b$. We have a diagram \minCDarrowwidth.1cm
\eq{excision}{\begin{CD}0 @>>>  H_1(V_0,\sL) @>>> H_1(V_0,p;\sL) @>>> H_0(p,\sL)  @>>> H_0(V_0,\sL)@>>> 0 \\
@. @VVV @V\cong V \text{excision} V @VV\alpha V \\
@. H_1(V,\sL) @>>> H_1(V,V_c;\sL) @>> b > H_0(V_c,\sL) 
\end{CD}
}
With the help of Lemma~\ref{free-H-1-lemma}, the top line is identified with the exact sequence
$$0 \to   \sL_p\xrightarrow{\sigma_0-1} \sL_p \to \sL_p/(\sigma_0-1)\sL_p \to 0.
$$
(To see this, one can, for example, think of $H_1$ as being given by $1$-chains coupled to sections of $\sL$. Chains with boundary at $p$ yield relative homology classes. Note $\sigma_0$ involves multiplication by $\exp(2\pi is)$, so $\sigma_0-1$ is injective.) The map $\alpha: \sL_p \to \sL_p/(\sigma_c-1)\sL_p$ is the evident one induced from the inclusion $p\in V_c$. A diagram chase identifies $Ker(b)$ with the kernel of the composition
$$\sL_p \xrightarrow{\sigma_0-1} \sL_p \to \sL_p/(\sigma_c-1)\sL_p 
$$
which is $ (\sigma_0-1)\sL_p\cap (\sigma_c-1)\sL_p\subset  (\sigma_c-1)\sL_p$.  Write $I_c:=\text{Image}(\sigma_c-1: \sM^\vee_p \to \sM^\vee_p)$. By assumption $I_c$ is a vector space of dimension $d$ over $K$, and we have 
$$(\sigma_c-1)\sL_p = I_c\otimes_K K[e^{\pm 2\pi is}],
$$ 
a $K[e^{\pm 2\pi is}]$-module of rank $d$. Since $\sigma_0-1$ is an injective map of $K[e^{\pm 2\pi is}]$-modules, it follows that $\ker(b)$ is a $K[e^{\pm 2\pi is}]$-module of rank $d$. \end{proof}

From now on we will focus on the case when Lemma~\ref{gamma-generators-prop} guarantees there is a unique generator of the module of gamma functions.  

\begin{defn}\label{c-assum} A regular singular point $c$ of $M$ will be called a \emph{reflection point} if the variation of the local monodromy of $\sM^\vee$ around $c$ has one dimensional image.   

At the base point $p$ we fix a non-zero solution $\delta \in \sM^\vee_p$ spanning the image $(\sigma_c-1)\sM^\vee_p$.
\end{defn}

\begin{ex}[vanishing cycle] A conifold point\footnote{the term {\it conifold point} seems to have passed from string physics to math. We will take it to refer to a fibre which contains a non-degenerate double point and no other singularities, \cite{G}.} of a family of algebraic varieties provides an example of a reflection point, in which case $\sM$ is self-dual and one can take $\delta$ to be the vanishing cycle at $c$. By the Picard--Lefschetz theorem the variation of the monodromy around $c$ satisfies
\[
(\sigma_c-1)\ve = \pm \la \ve, \delta \ra \delta
\]
for any section $\ve \in \sM$. When fibres of the family have even dimension we have $\la \delta, \delta \ra=\pm 2$, $\sigma_c$ is semisimple on $\sM$ with $\sigma_c^2=1$ and $\sigma_c \delta = - \delta$. In the case of odd dimensional fibres one has $\la \delta, \delta \ra =0$ and $\sigma_c \delta = \delta$.
\end{ex}

\begin{lem}\label{annlem} Let $c \ne 0,\infty$ be a reflection point for $M$ as in Definition~\ref{c-assum}. Let $V \subset U_{an}$ be a neighbourhood of a path between $0$ and $c$ satisfying Assumption~\ref{V-assum}. Let $\ve \in \sM_p^\vee$ be such that $(\sigma_c-1)\ve = \delta$. For every relation 
\eq{rels}{(\sum_m\lambda_m \sigma_0^{m})\delta = 0 \quad (\text{ a finite sum with } \lambda_m \in K)
}
the element 
\eq{xi}{\xi = \sum_m \lambda_m\sigma_0^m\otimes\delta\otimes e^{-2\pi ims}+\sigma_c\otimes\ve\otimes \sum_m \lambda_m e^{-2\pi ims}
}
is a $1$-cycle with coefficients in $\sM_p^\vee \otimes_K K[e^{\pm 2 \pi i s}]$. The resulting map to the homology of the local system $\sL = \sM^\vee \otimes t^s$
\eq{ann}{ \text{Ann}_{K[\sigma_0^{\pm 1}]}(\delta) \to H_1(V,\sL)
}
 is a homomorphism of $K[T^{\pm 1}]$-modules, where $T$ acts via multiplication by $\sigma_0$ on relations \eqref{rels} and by multiplication by $e^{2\pi is}$ on homology. The group $H_1(V, \sL)$ is spanned by the images of $H_1(V_c,\sL)$ and \eqref{ann}. 
\end{lem}
\begin{proof}To check that $\xi$ is a cycle, we compute
\[\bal
\partial\xi = \sum_m \lambda_m(\sigma_0^{m}-1)(\delta\otimes e^{-2\pi ims}) + (\sigma_c-1)\ve \otimes \sum_m \lambda_m e^{-2\pi ims} \\
= \sum_m \lambda_m \sigma_0^{m}(\delta\otimes e^{-2\pi i s m}) = (\sum_m \lambda_m \sigma_0^{m}\delta)\otimes 1 = 0. 
\eal\]

The $K[T^{\pm 1}]$-structure is straightforward and left for the reader. 

It remains to check that the map \eqref{ann} is surjective modulo the image of $Cor: H_1(V_c,\sL) \to H_1(V,\sL)$. Since $\pi_1(V,p)$ is a free group generated by $\sigma_0$ and $\sigma_c$, every $1$-cycle can be written modulo boundaries in the form 
\eq{chain}{\sigma_0\otimes \Big(\sum_n \psi_n\otimes e^{2\pi ins}\Big) + \sigma_c\otimes \Big(\sum_m \gamma_m\otimes  e^{2\pi ims}\Big).
}
Here $\psi_n,\gamma_m \in \sM_p^\vee$. (The point is that modifying by a boundary can remove any words in the $\sigma_c$ and $\sigma_0$. See equation \eqref{2-boundary}.) We write each $\gamma_m=\lambda_{-m}\ve + \gamma_m^{inv}$ with $\lambda_{-m} \in K$ and $\gamma_m^{inv} \in (\sM^\vee_p)^{\sigma_c=id}$. The chain \eqref{chain} is a sum of
$$\sigma_c\otimes \Big(\sum_m \gamma_m^{inv}\otimes e^{2\pi ims}\Big)
$$
which is itself a $1$-cycle on the subgroup generated by $\sigma_c$, and 
\eq{tilde}{\tilde\xi := \sigma_0\otimes\Big(\sum_{n=n_{min}}^{n_{max}}\psi_n\otimes e^{2\pi ins}\Big) + \sigma_c\otimes \ve\otimes \sum_m \lambda_{-m} e^{2\pi i m s}.
}
The $1$-cycle condition yields
\[0 = \partial\tilde\xi = \sum_n (\sigma_0\psi_{n-1}-\psi_n+\lambda_{-n}\delta)\otimes e^{2\pi ins}. 
\]
(Recall the action of $\sigma_0$ includes multiplication by $e^{2\pi is}$, \eqref{Mellin-stalk}.) This equation can be solved recursively
\[\bal
&\,\lambda_{-n} \;=\; 0,\ n<n_{min} \\ 
&\psi_{n_{min}} \;=\; \lambda_{-n_{min}}\delta \\
&\psi_{n_{min}+1} \;=\; \lambda_{-n_{min}}\sigma_0\delta+\lambda_{-n_{min}-1}\,\delta \\
&\qquad \vdots \\
0 \;=\; &\psi_{n_{max}+1} \;=\; \sum_{j=0}^{n_{max}-n_{min}+1}(\lambda_{-n_{min}-j}\,\sigma_0^{(n_{max}-n_{min}+1)-j})\delta
\eal\]
It follows that $\Big(\sum_m \lambda_m \sigma_0^m\Big)\delta = 0$. In fact, \eqref{tilde} is homologous to \eqref{xi}. To check this, we can assume $\lambda_m=0$ for $m\le 0$. For $v\in \sL_p$ we have by \eqref{2-boundary}
$$\sigma_0^m\otimes v \sim \sigma_0\otimes (1+\cdots+\sigma_0^{m-1})v,
$$
and therefore
$$\sum_m \lambda_m\sigma_0^m\otimes \delta\otimes e^{-2\pi ims} \sim \sigma_0\otimes v,
$$
with
\[\bal v &= \sum_{m>0} \lambda_m(1+\sigma_0+\cdots + \sigma_0^{m-1})(\delta\otimes e^{-2\pi ims})  = \sum_{m>0} \lambda_m\sum_{j=0}^{m-1} \sigma_0^{j}\delta\otimes e^{2\pi i(j-m)s} \\
& = \sum_{n<0}\Big(\sum_{m \ge -n}\lambda_m\sigma_0^{m+n}\delta\Big)\otimes   e^{2\pi ins} = \sum_n \psi_n\otimes e^{2\pi ins}. 
\eal\]
\end{proof}

With the help of this lemma, we can now compute the generator of the module of gamma functions. We will return to this computation again in the following sections, where we combine it with the duality for connections to show a relation between generalized gamma functions and monodromy of Frobenius solutions. As in Section~\ref{sec:motivic-gamma}, we fix an element $m \in M$ and consider gamma functions corresponding to $m \otimes dt/t \in M \otimes \Omega^1_U$ (Definition~\ref{gamma-def-0}).

\begin{prop}\label{unique-gamma} Let $c \ne 0,\infty$ be a reflection point for $M$ as in Definition~\ref{c-assum}. Let $V \subset U_{an}$ be a neighbourhood of a path between $t=0$ and $t=c$ satisfying Assumption~\ref{V-assum}. Let $R \in K[\sigma_0]$ be a polynomial of minimal degree such that $R(\sigma_0) \delta = 0$.

Possibly multiplying $m \in M$ by a power of $(t-c)$, we assume that functions $\langle m , \varepsilon \rangle$ are analytic at $t=c$ for every $\sigma_c$-invariant solution $\ve \in \sM^\vee$ and that $\langle m,\delta \rangle$ is $O(|t-c|^{\alpha-1})$ as $t \to c$ for some $\alpha>0$. Then the $K[e^{\pm 2 \pi i s}]$-module of gamma functions
\[
\Gamma_\xi(s), \quad \xi \in H_1(V, \sM^\vee \otimes t^s)
\]
is generated by 
\eq{delta-int-0-c}{
\Gamma_{\xi_0}(s) = R(e^{-2 \pi i s}) \int_0^c  \langle m, \delta \rangle t^{s-1} dt.
}
\end{prop}

Integration in the right-hand side of~\eqref{delta-int-0-c} is done along our chosen path (of which $V$ is a neighbourhood). Note that, since $0$ at $c$ are singular points, it is possible that this integral is improper. The growth condition on $\langle m,\delta \rangle$ as $t \to c$ ensures convergence at this endpoint. The integral is also convergent at the other endpoint $t=0$ whenever $Re(s)$ is sufficiently large. Hence the the right-hand side in~\eqref{delta-int-0-c} is defined when $Re(s) \gg 0$, and we claim that it actually extends to an entire function in the complex plane, which is our generalized gamma function.

\begin{proof}[Proof of Proposition~\ref{unique-gamma}.]The argument given at the beginning of this section shows that, since functions $\langle m , \varepsilon \rangle$ are analytic at $t=c$ for every $\sigma_c$-invariant $\ve \in \sM_p^\vee$, we have $\Gamma_\xi(s) \equiv 0$ for any class  $\xi$ in the image of $H_1(V_c,\sM^\vee \otimes t^s)$. Now let $\ve \in \sM_p^\vee$ be such that $(\sigma_c-1)\ve = \delta$. For a polynomial $P(T)=\sum_k \lambda_k T^k$ such that $P(\sigma_0)\delta = 0$ we compute the gamma function corresponding to the homology class of~\eqref{xi} as follows:
\[\bal
\Gamma_{\xi}(s) &= \sum_k \lambda_k e^{-2 \pi i k s} \int_{\sigma_0^{k}} \langle m, \delta \rangle t^{s-1} dt + P(e^{-2 \pi i s})\int_{\sigma_c} \langle m, \ve  \rangle t^{s-1} dt \\
&= \sum_k \lambda_k e^{-2 \pi i k s} \int_0^p \langle m,  (e^{2\pi i k s} \sigma_0^{k} - 1)\delta \rangle   t^{s-1} dt  + P(e^{-2 \pi i s}) \int_c^p  \langle m, (\sigma_c-1)\ve \rangle   t^{s-1} dt \\
&= \int_0^p \langle m, P(\sigma_0)\delta \rangle - P(e^{2 \pi i s}) \int_0^c  \langle m, \delta \rangle t^{s-1} dt \;=\;  - P(e^{-2 \pi i s}) \int_0^c  \langle m, \delta \rangle t^{s-1} dt.
\eal\]
The statement of the proposition follows if we take $P(T)=-R(T)$.
\end{proof}

\begin{ex}[polylogarithm]\label{polylog-example} For an integer $n \ge 1$, the $n$th polylogarithm is a multivalued holomorphic function, one of whose branches in the open unit circle $|t|<1$ is given by the convergent series $Li_n(t) = \sum_{k=1}^\infty k^{-n} t^k$. One can easily see that this function is annihilated by the differential operator
\[
L = \left((1-t)D - 1\right) \, D^n.
\]
The operator $L$ has regular singularities and the local system of its solutions on $U=\P^1\setminus \{0,1,\infty\}$ is spanned by $Li_n(t)$ and $\log^k(t)$ for $0 \le k \le n-1$. The singularity at $t=1$ is a reflection point, and we can take
\[
\delta(t) := (\sigma_1-1) Li_n(t) = - 2 \pi i \frac{\log^{n-1}(t)}{(n-1)!},
\] 
see e.g.~\cite[Proposition 2.2]{Hain}. The annihilator of $\delta(t)$ in $\C[\sigma_0^{\pm 1}]$ is generated by $(\sigma_0-1)^n$ and hence the $\C[e^{\pm 2 \pi i s}]$-module of gamma functions for the $n$th polylogarithm is generated by
\[
(1-e^{2 \pi i s})^n \int_0^1 \delta(t) t^{s-1} dt = -2 \pi i \left(\frac{e^{2 \pi i s}-1}{s}\right)^n.
\] 
\end{ex}

\bigskip

Note that for connections on $U=\P^1\setminus\{0,c,\infty\}$ Lemma~\ref{gamma-generators-prop} applies with $V=U$. It follows that there is a unique gamma function attached to every hypergeometric connection, of which Example~\ref{polylog-example} is a degenerate case. We will consider hypergeometric connections in Example~\ref{hg-example}.

\section{Frobenius constants}\label{sec:Apery}

In this section we consider an algebraic connection $M$ of rank $r$ on $U = \P^1\setminus S$, where $S$ is a finite set of points. The local system of solutions $\sM^\vee = (M_{an}^\vee)^{\nabla^\vee=0}$ is defined over a field $K \subseteq \C$. We assume that $c \in S$ is a~\emph{reflection point} for $M$ (Definition~\ref{c-assum}). We remind the reader, this means that $c$ is a regular singularity and the variation of the local monodromy of solutions around $c$ has one-dimensional image. For any other regular singularity $c' \in S$ and a homotopy class of paths $\gamma$ going from $c'$ to $c$ through the points of $U$, \emph{Frobenius constants} describe the variation around $c$ of the Frobenius solutions near $c'$. Frobenius solutions are sections in $\sM^\vee \otimes_K \C $ produced by the classical Frobenius method in the theory of differential equations. For geometric (Gau\ss--Manin) connections Frobenius solutions can be used to describe the limiting mixed Hodge structure at $c'$ (in particular, they span the limiting de Rham structure, see Section~\ref{sec:Frob-near-MUM}). We shall start with recalling the Frobenius method. It will be convenient to assume that $c'=0$ throughout this section. 

Let $\sO_U \subset k(t)$ be the ring of rational functions regular on $U$. The ring of differential operators $\sD = \sD_U$ is generated over $\sO_U$  by the derivation $D=t \frac{d}{dt}$. We fix an element $m \in M$ and a differential operator
\eq{L-diff-op}{
L = q_0(t) D^r + q_1(t) D^{r-1} + \ldots + q_{r-1}(t) D + q_r(t), \quad q_j \in \sO_U, 
}
such that $Lm=0$. Let $Sol(L)$ be the local system of solutions of $L$ on $U$. It is always possible to choose $m$ so that the elements $m, Dm, D^2m, \ldots$ generate $M$ as an $\sO_U$-module (see~\cite[\S 2.3.1]{Haefliger}). With such a choice, we have $q_0 \in \sO_U^\times$ and  the pairing $\langle m , * \rangle$ identifies $\sM^\vee \otimes_K \C$ with the local system $Sol(L)$. When $K \ne \C$, this identification gives a $K$-structure on the local system $Sol(L)$, which we will sometimes refer as $Sol(L)_K \cong \sM^\vee$.  

We assume that 
\[
q_0(0) \ne 0.
\]
In this case, $t=0$ is at most a regular singularity if and only if none of the coefficients $q_j(t)$ has a pole at $t=0$. We assume this is the case. The~\emph{indicial polynomial} of $L$ at $t=0$ is defined as
\[
I(s) = q_0(0) s^r + q_1(0) s^{r-1} + \ldots + q_r(0).
\]
The roots of $I(s)$ are called the~\emph{local exponents} of $L$ at $t=0$. We will look for solutions of $L$ in the form $t^s \sum_{n=0}^\infty a_n t^n$.  Expanding the coefficients $q_j(t)$ in Taylor series at $t=0$ and collecting the terms in powers of $t$, we write $L=\sum_{j=0}^\infty t^j p_j(D)$ where $p_j \in k[D]$ are polynomials of degree $\le r$. Note that $p_0(s)=I(s)$ is the indicial polynomial. We then have $L \left( \sum_{n \ge 0} a_n t^{n+s} \right) = \sum_{n \ge 0} t^{n+s} \Bigl( p_0(n+s) a_n + p_1(n+s-1) a_{n-1} + \ldots \Bigr)$. Consider the sequence of rational functions $a_n \in k(s)$ uniquely defined by the conditions
$a_n = 0$ when $n<0$, $a_0=1$ and $\sum_{j=0}^n p_j(n+s-j) a_{n-j}(s) = 0$ for all $n \ge 1$. By construction, the formal series 
\eq{Frob-def-0}{
\Phi(s,t) =  \sum_{n=0}^\infty a_n(s) t^{n+s}
}
satisfies the inhomogeneous differential equation 
\eq{Frob-def-de}{
L \Phi = I(s) t^s.
}
Note also that the denominator of $a_n(s)$ divides $I(s+1)\ldots I(s+n)$. Consider the set of local exponents $\rho$ such that none of $\rho+1, \rho+2, \ldots$ is a local exponent:
\eq{rho-set}{
\sR = \{ \rho \in \C \;|\; I(\rho) = 0, \; I(\rho + n) \ne 0 \text{ for all } n \in \Z_{\ge 1}\}.
}
If the mutiplicity of a root $\rho \in \sR$ is $m$, it follows from~\eqref{Frob-def-de} that functions
\[\bal
&\phi_{\rho,0}(t) = \Phi(\rho, t) = \sum_{n=0}^\infty a_n(\rho) t^{n+\rho}, \\
&\phi_{\rho,1}(t) = \frac{\partial \Phi (s,t)}{\partial s} \Big|_{s=\rho} = \log(t) \sum_{n=0}^\infty a_n(\rho) t^{n+\rho} + \sum_{n=0}^\infty a_n'(\rho) t^{n+\rho},\\
&\ldots \\
&\phi_{\rho,m-1}(t) = \frac1{(m-1)!}\frac{\partial^{m-1} \Phi(s,t)}{\partial s^{m-1}} \Big|_{s=\rho} =  \sum_{j=0}^{m-1} \frac{\log(t)^j}{j!} \sum_{n=0}^\infty \frac{a_n^{(m-1-j)}(\rho)}{(m-1-j)!} t^{n+\rho}  \\ 
\eal\]  
are solutions of $L$:
\[
L \phi_{\rho,k} = 0, \quad 0 \le k < m.
\]
We did the above computation formally, but it can be shown that the series that occur in these formulas converge in a neighbourhood of $t=0$. The Frobenius method gives actual analytic solutions of $L$ for every choice of a branch of $t^s$.

\begin{remark}\label{local-exponents-and-monodromy-rmk} It is clear from the above formulas that the solution $\phi_{\rho,0}$ is an eigenvector of the local monodromy operator $\sigma_0$. Namely, we have $\sigma_0 \phi_{\rho,0} = e^{2 \pi i \rho} \phi_{\rho,0}$. It follows that, for every local exponent $\rho$, the number $e^{2 \pi i \rho}$ is an eigenvalue of $\sigma_0$. In the case when $I(s)$ has no roots that differ by a non-zero integer, solutions constructed by the Frobenius method give a basis in the vector space $Sol_p(L)$ of solutions that are defined in a neighbourhood of a regular point $t=p$ located sufficiently close to $t=0$. One can easily deduce this from our construction because $\dim_\C Sol_p(L) = r$ by the classical theorem of Cauchy. When $I(s)$ has roots differing by a non-zero integer, one can not construct a basis in the space of solutions by the Frobenius method. In this case it is possible to transform the differential equation so that the above condition holds. For that one performs a sequence of \emph{shearing transformations} which modify local exponents by integers, see~\cite[\S1.4.5]{Haefliger}. 

For our purposes, we do not need to assume that Frobenius solutions from a basis. That is, it can be that $\sR$ is smaller than the full set of local exponents $\{ \rho | I(\rho)=0 \}$.    
\end{remark}

A novel idea in~\cite{GZ} is to continue differentiating~\eqref{Frob-def-de} and substituting $\rho$, even though we no longer obtain solutions of $L$. The collection of \emph{Frobenius functions} is defined as
\eq{Frob-functions}{
\phi_{\rho,k}(t) = \frac1{k!}\frac{\partial^{k} \Phi(s,t)}{\partial s^{k}} \Big|_{s=\rho} =  \sum_{j=0}^{k} \frac{\log(t)^j}{j!} \sum_{n=0}^\infty \frac{a_n^{(k-j)}(\rho)}{(k-j)!} t^{n+\rho}, \quad \rho \in \sR, \; k \ge 0.
}
When $k \ge m(\rho)$ Golyshev and Zagier call~\eqref{Frob-functions} \emph{higher Frobenius functions}. It follows from~\eqref{Frob-def-de} that the higher functions satisfy differential equations
\eq{DL}{
(D-\rho)^{k-m+1} L \phi_{\rho,k} = 0, \quad k \ge m. 
}
We will often use the term \emph{Frobenius solutions} rather then \emph{functions}, having in mind that the higher $\phi_{\rho,k}$ satisfy inhomogeneous differential equations $L \phi_{\rho,k} = \ldots$ with the right-hand side being a polynomial in $\log(t)$ which can be written down explicitly from~\eqref{Frob-def-de}.  

\bigskip

We shall now consider analytic continuation of the Frobenius solutions along a path $\gamma$ joining $t=0$ with a reflection point $t=c$. We choose a base point $p \ne 0,c$ on the path $\gamma$ and work in the stalk $Sol_{p}(L)$, which is the $\C$-vector space  of solutions of $L$ in a neighbourhood of $p$. By Cauchy's theorem $\dim_\C Sol_{p}(L) = r$. On this vector space we have the operators of local monodromy $\sigma_0$ and $\sigma_c$ around $0$ and $c$ respectively. We fix a solution $\delta(t) \in Sol_{p}(L)$ such that 
\[
Image \Bigl( \sigma_c - 1 \Big| Sol_{p}(L) \Bigr) = \C \, \delta.  
\] 
It means that every solution of $L$ adds a constant multiple of $\delta(t)$ when going around $t=c$. 
{ The crucial point for defining Frobenius constants is that, under a mild condition, the same property remains true for the higher Frobenius functions (Lemma~\ref{c-remains-reflection}). 
In order to state this condition, let us consider the dual connection $M^\vee$. The~\emph{adjoint} of the  differential operator~\eqref{L-diff-op} is defined by the formula
\eq{L-adjoint}{
L^\vee = (-D)^r q_0(t) + (-D)^{r-1} q_1(t) +\ldots + q_r(t).
} It is possible to choose a generator of $M^\vee$ so that it is annihilated by $L^\vee$ (see \cite[Gabber's Lemma in \S 1.5]{Katz87}.) For the sake of completeness, we give a proof of this fact in Section~\ref{sec:Frob-def-and-gamma}. In particular, there is a perfect monodromy-invariant pairing on solutions 
\[
\{*,*\}: Sol(L^\vee) \otimes_\C Sol(L) \to \C
\]
which will be constructed explicitly in the next section,~\eqref{bracket}. 

\begin{lem}\label{L-adj-refl} If $t=c$ is a reflection point of $L$, then it is also a reflection point of $L^\vee$.    
\end{lem}
\begin{proof} Since $\{ (\sigma_c - 1) \psi, \phi \} = \{ \psi, (\sigma_c^{-1}-1) \phi\}$, a solution $\phi \in Sol_{p}(L)$ is orthogonal to the image of $Sol_{p}(L^\vee)$ under $(\sigma_c-1)$ if and only if $\sigma_c \phi = \phi$. Since the duality pairing is perfect and $\dim Ker( \sigma_c -1 | Sol_{p}(L)) = r-1$, it follows that $\dim Image( \sigma_c -1 | Sol_{p}(L^\vee))=1$.   
\end{proof}

\begin{lem}\label{c-remains-reflection} Assume that $c \ne 0,\infty$ is a reflection point for $L$. Possibly dividing $L$ by a power of $(t-c)$ on the left, we can assume that all $\sigma_c$-invariant solutions of the adjoint differential operator $L^\vee$ are analytic near $t=c$. Then $c$ is also a reflection point for the differential operators $(D-\rho)^jL$ with any $\rho \in \C$ and any integer $j \ge 1$.
\end{lem}

\begin{proof} By Lemma~\ref{L-adj-refl}, $c$ is a reflection point of $L^\vee$. Hence $L^\vee$ has a codimension~1 subspace of solutions that are $\sigma_c$-invariant or, equivalently, meromorphic functions near $t=c$. Let $a \ge 0$ be the maximal order of pole at $t=c$ for such solutions. Then all $\sigma_c$-invariant solutions of $((t-c)^{-a}L)^\vee = L^\vee (t-c)^{-a}$ are analytic at $t=c$.

We shall now assume that the condition for $L^\vee$ holds and prove the second statement. It is clear that  $\dim Image \Bigl( \sigma_c - 1 \Big| Sol_{p}((D-\rho)^j L) \Bigr) \ge 1$ because this space can only get larger when we increase $j$. Therefore it suffices to show that for each $j \ge 1$ the adjoint operator $((D - \rho)^j L)^\vee = (-1)^j L^\vee (D+\rho)^j$ has a codimension~1 subspace of solutions analytic near $t=c$. Clearly, we can do the case $j=1$ and apply induction on $j$. Let $\phi \in Ker(\sigma_c-1 | Sol_{p}(L^\vee))$. Then any $\phi'$ such that $(D+\rho)\phi'=\phi$ can be recovered (up to adding a constant multiple of $t^{-\rho}$) as $\phi'=t^{-\rho} \int t^{\rho-1} \phi(t) dt$. By our assumption, $\phi$ is an analytic function near $t=c$ and hence $\phi'$ will be also analytic near $t=c$. It follows that the space of solutions of $L^\vee(D+\rho)$ analytic near $t=c$ is at least one dimension bigger than the same space for $L^\vee$. 
\end{proof}

\begin{ex} The differential operator $L= D(D+1) - t(D+1)^2$ has a reflection point at $t=1$. Solutions of the adjoint operator $L^\vee = (D-1)D - tD^2$ are spanned by $1$ and $\log(1-t)$ (this is the case $n=1$ of Example~\ref{polylog-example}). We see that $t=1$ is a reflection point and, by Lemma~\ref{c-remains-reflection},  it will remain a reflection point for all operators $(D-\rho)^j L$ with $\rho \in \C$ and $j \ge 1$. 

To see that the condition on $L^\vee$ in Lemma~\ref{c-remains-reflection} is important, let us multiply the above differential operator by $1-t$ on the left. We obtain $L = D(D+1)-t(2 D+1)(D+1) + t^2(D+1)^2$. The reader may check that the monodromy of the 3rd order operator $DL$ around $t=1$ is maximally unipotent. It follows that $\dim Image (\sigma_1-1 | Sol_p(DL))=2$, so $t=1$ is not a reflection point for $DL$.

\end{ex}

\begin{defn}\label{strong-reflection-pt-def} We say that a regular singularity $t=c$ of a differential operator $L$ is a \emph{special reflection point} if the variation of the local monodromy of solutions of $L$ around $c$ has one-dimensional image and all $\sigma_c$-invariant solutions of the adjoint differential operator $L^\vee$ are analytic at $c$. 
\end{defn}

If $t=c$ is a special reflection point then  the image of $\sigma_c-1$ on the bigger spaces $Sol_{p}((D-\rho)^j L)$ is spanned by $\delta(t)$ for all  $j > 0$ by Lemma~\ref{c-remains-reflection}.

\begin{defn}\label{full-apery-defn} Assume that $c \ne 0,\infty$ is a \emph{special reflection point} of $L$. Let $\gamma$ be a path from $0$ to $c$ going through regular points of $L$. Fixing a branch of $t^s$ along $\gamma$, we have a collection of Frobenius functions $\{ \phi_{\rho,n}(t)\}_{\rho \in \sR, n \ge 0}$ defined by the analytic continuation of~\eqref{Frob-functions} along $\gamma$. The collection of \emph{Frobenius constants} $\{\kappa_{\rho,n}  \}_{\rho \in \sR, n \ge 0}$ is defined by
\[
(\sigma_c - 1) \phi_{\rho,n}(t) = \kappa_{\rho,n} \, \delta(t), \quad \kappa_{\rho, n} \in \C.
\]   
\end{defn} 

\bigskip

It is clear that the Frobenius constants in this definition depend on the homotopy class of the path $\gamma$. Note that changing $\delta(t) \mapsto \lambda \delta(t)$, $\lambda \in \C^\times$ will result in the collection of Frobenius constants being rescaled as $\{ \kappa_{\rho,n} \} \mapsto \{ \lambda^{-1} \kappa_{\rho,n}\}$. The reader may also check the effect on $\{ \kappa_{\rho,n} \}$ of choosing a different branch of $t^s$ in the definition of the Frobenius solutions.

\begin{ex}[hypergeometric differential equations]\label{hg-example} Fix some numbers $\alpha_1,\ldots,\alpha_r$, $\beta_1,\ldots,\beta_r \in \R$ and consider the differential operator 
\eq{hg-diff-op}{
L = \prod_{j=1}^r (D+\beta_j-1) - t \prod_{j=1}^r (D+\alpha_j)
}
on $U = \P^1 \setminus \{0,1,\infty\}$. This operator is already written in the form $L = p_0(D) + t p_1(D)$, and to apply the Frobenius method near $t=0$ we notice that the function
\eq{hg-A}{
A(s) = \frac{\prod_{j=1}^r \Gamma(s + \alpha_j)}{\prod_{j=1}^r \Gamma(s + \beta_j)}
}
satisfies the difference equation $p_0(s) A(s) + p_1(s-1)A(s-1)=0$. It then follows that the coefficients $a_n(s)$ in~\eqref{Frob-def-0} are given by
\eq{hg-Frob-recurrence}{
a_n(s) = \frac{A(n+s)}{A(s)} = \frac{\prod_{j=1}^r (s + \alpha_j)_n}{\prod_{j=1}^r (s + \beta_j)_n},
}
where $(s)_n = \Gamma(n+s)/\Gamma(s)$ is the so-called Pochhammer symbol. The subset of local exponents~\eqref{rho-set} is given by $\sR = \{ 1-\beta_i | \beta_j - \beta_i \not\in \Z_{< 0} \text{ for all } j\}$.

The global monodromy representation corresponding to the local system of solutions of~\eqref{hg-diff-op} is irreducible if and only if $\alpha_j - \beta_i \not\in \Z$ for any pair of indices $j,i$. Under this assumption, $t=1$ is a special reflection point: the conditions of Definition~\ref{strong-reflection-pt-def} are satisfied by \cite[Propositions 2.8 and 2.10]{B-H}.  Consider the direct path from $t=0$ to $t=1$ along the real line. We will see (Proposition~\ref{hg-kappa-A-prop}) that the respective collection of Frobenius constants $\{ \kappa_{\rho,n}\}_{\rho \in \sR, n \ge 0}$ is given by the expansion coefficients of the function $A(s)^{-1}$. Namely, we have
\[
\frac1{A(s)} = \sum_{n=0}^{\infty} \kappa_{\rho,n} (s - \rho)^n, \quad  \rho \in \sR.
\]
Note that the irreducibility condition $\alpha_j - \beta_i \not\in \Z$ implies that $A(\rho) \in \R^\times$ for each $\rho \in \sR$. 
\end{ex}

Let us describe a method of computation of Frobenius constants in the case when the reflection point under consideration is the closest singularity to $0$. 

\begin{lem}\label{kappa-lambda-lem} Suppose that the special reflection point $t=c$ is the closest to $t=0$ singularity of a differential operator $L$; suppose in addition that $c \in \R_{>0}$. Let $\{ \kappa_{\rho,n}\}_{\rho \in \sR, n \ge 0}$ be a collection of Frobenius constants for the direct path from $0$ to $c$ along the real line.  Fix a local exponent $\rho \in \sR$. Assume further that
\begin{itemize}
\item[(i)] $|\phi_{\rho,0}(t)| \to \infty$ as $t \to c_{-}$;

\item[(ii)] all $\sigma_c$-invariant solutions of operators $(D-\rho)^j L$, $j \ge 0$ are analytic at $c$; 

\item[(iii)] coefficients $a_n(\rho)$ are real numbers of the same sign when $n \gg 0$;
 
\item[(iv)] for each $k \ge 1$ there exists a finite limit $\lambda_k := \lim_{n \to \infty} \frac{a_n^{(k)}(\rho)}{k! \, a_n(\rho)}$.
\end{itemize} 
Then $\kappa_{\rho,0} \ne 0$ and we have
\[
\frac{\kappa_{\rho,k}}{\kappa_{\rho,0}} = \sum_{j=0}^k \frac{\log(c)^{j}}{j!} \lambda_{k-j}.
\]
\end{lem}
\begin{proof}
Using power series $\phi_{\rho,k}^{an}(t):= \frac1{k!} \sum_{n \ge 0} a_n^{(k)}(\rho) \, t^n$, we can write 
\[
\phi_{\rho,k}(t) = t^\rho \sum_{j=0}^k \frac{\log(t)^j}{j!} \phi^{an}_{\rho,k-j}(t).
\]
Conditions~(i) and~(ii) ensure that $\phi_{\rho,0}$ is not $\sigma_c$-invariant, and hence $\kappa_{\rho,0} \ne 0$. Observe that $\phi_{\rho,k}(t)- \frac{\kappa_{\rho,k}}{\kappa_{\rho,0}} \phi_{\rho,0}(t)$ is a $\sigma_c$-invariant solution of $(D-\rho)^j L$ for some $j$ (see~\eqref{DL}). The following computation can be done inductively for $k \ge 1$, which shows that the ratios $\frac{\phi_{\rho,k}^{an}(t)}{\phi_{\rho,0}^{an}(t)}$ have finite limits as $t \to c_{-}$:
\[\bal
0 & = \lim_{t \to c_{-}} \frac{\phi_{\rho,k}(t)- \frac{\kappa_{\rho,k}}{\kappa_{\rho,0}} \phi_{\rho,0}(t)}{\phi_{\rho,0}(t)} \quad \text{ ( by~(i) and~(ii))}\\
 & = \lim_{t \to c_{-}} \sum_{j=0}^k \frac{\log(t)^j}{j!} \frac{\phi_{\rho,k-j}^{an}(t)}{\phi_{\rho,0}^{an}(t)} - \frac{\kappa_{\rho,k}}{\kappa_{\rho,0}} = \sum_{j=0}^k \frac{\log(c)^j}{j!} \lim_{t \to c_{-}} \frac{\phi_{\rho,k-j}^{an}(t)}{\phi_{\rho,0}^{an}(t)}  - \frac{\kappa_{\rho,k}}{\kappa_{\rho,0}}.
\eal\]
It remains to show that $\lim_{t \to c_{-}} \frac{\phi_{\rho,k}^{an}(t)}{\phi_{\rho,0}^{an}(t)} = \lambda_k$.  It follows from~(i) that $\phi_{\rho,0}^{an}(t) = t^{-\rho} \phi_{\rho,0}(t) = \sum_n a_n(\rho) t^n$ also grows infinitely as $t \to c_{-}$. For any $\ve > 0$ there is an $N$ such that $|\frac{a_n^{(k)}(\rho)}{k! \, a_n(\rho)}-\lambda_k| < \ve$ for all $n > N$. Thus 
\[\bal
& \left| \lim_{t \to c_{-}} \frac{\phi_{\rho,k}^{an}(t)}{\phi_{\rho,0}^{an}(t)} - \lambda_k\right| = \lim_{t \to c_{-}} \left| \frac{\phi_{\rho,k}^{an}(t) - \lambda_k \phi_{\rho,0}^{an}(t)}{\phi_{\rho,0}^{an}(t)}\right| =   \lim_{t \to c_{-}} \left| \frac{\sum_{n > N} (\frac{ a_n^{(k)}(\rho)}{k!} - \lambda_k a_n(\rho))t^n}{\phi_0^{an}(t)}\right| \\
& \text{ (here we used~(ii) to get rid of the polynomial part of the numerator )}\\
& \le  \lim_{t \to c_{-}} \frac{\ve  \sum_{n > N}|a_n(\rho)|\cdot  |t|^n}{|\phi_0^{an}(t)|} = \ve \quad \text{ (using (iii))}.\\
\eal\]
Since $\ve > 0$ is arbitrary, our claim follows.
\end{proof}

To apply the above lemma to hypergemetric differential operators, let us compute the limits $\lambda_k$ from part~(iv):  

\begin{lem}\label{A-lambda-lem} Consider the sequence of rational functions $a_n(s)= \frac{A(n+s)}{A(s)}$, where $A(s)$ is the gamma product~\eqref{hg-A} with real parameters $\alpha_1,\ldots,\alpha_r$, $\beta_1, \ldots, \beta_r$ satisfying the condition $\alpha_i - \beta_j \not\in \Z$, $\forall i,j$. 
Let $\rho \in \sR$ be a local exponent of the respective hypergeometric differential operator~\eqref{hg-diff-op}  at $t=0$.  The limits $\lambda_k := \lim_{n \to \infty} \dfrac{a_n^{(k)}(\rho)}{k!\, a_n(\rho)}$ exist and coincide with expansion coefficients of $A(\rho)/A(s)$ at $s =\rho$:
\[
\frac{A(\rho)}{A(s)} = \sum_{k \ge 0} \lambda_k (s-\rho)^k.
\] 
\end{lem}
\begin{proof} Stirling's formula 
\[
\log \Gamma(s) = (s - \frac12) \log(s) - s + \frac12 \log(2 \pi) + O(\frac1{s}), \quad s \to \infty 
\]
implies that $\log A(s) = \left(\sum_i \alpha_i - \sum_j \beta_j \right) \log(s) + O(s^{-1})$. It follows that
\[
\lim_{n \to \infty} \frac{A(n+s)}{A(n+\rho)} = 1.
\]
Moreover, the convergence here is uniform for $s$ in a small disk about $\rho$ because the $O(s^{-1})$ term is Stirling's formula can be explicitly estimated (its absolute value  is known to be bounded by  $\frac1{12} |s|^{-1}$ when $|\arg(s)|<\frac14$). As we explained in Example~\ref{hg-example}, when $L$ is irreducible then $A(\rho) \ne 0$. Thus we have a uniform convergence of functions 
\[
\lim_{n \to \infty} \frac{a_n(s)}{a_n(\rho)} =  \frac{A(\rho)}{A(s)} \lim_{n \to \infty} \frac{A(n+s)}{A(n+\rho)} = \frac{A(\rho)}{A(s)}.
\] 
As these functions are holomorphic is a small disk about $\rho$, we can conclude that for every $k$ the $k$th derivative of $a_n(s)/a_n(\rho)$ at $s=\rho$ converges to the $k$th derivative of $A(\rho)/A(s)$ at $s=\rho$.   
\end{proof}

\begin{prop}\label{hg-kappa-A-prop} Consider the hypergeometric differential operator~\eqref{hg-diff-op} with real parameters $\alpha_1,\ldots,\alpha_r$, $\beta_1, \ldots, \beta_r$ satisfying the irreducibility condition $\alpha_i - \beta_j \not\in \Z$, $\forall i,j$. Then a collection of Frobenius constants  $\{ \kappa_{\rho,k}\}_{\rho \in \sR, k \ge 0}$ for the direct path from $0$ to $1$  is given by the expansion coefficients 
\eq{hg-kappa}{
\frac1{A(s)} = \sum_{n=0}^{\infty} \kappa_{\rho,n} (s - \rho)^n, \quad  \rho \in \sR,
}
where $A(s)$ is the gamma-product~\eqref{hg-A}.
\end{prop}

\bigskip

Below we give a proof of~\eqref{hg-kappa} for a fixed $\rho \in \sR$. As soon as this is done, it will remain to show that there is a normalization of $\delta(t)$ in Definition~\ref{full-apery-defn} such that $\kappa_{\rho,0}=A(\rho)^{-1}$ for every $\rho \in \sR$. Equivalently, one needs to show that the following linear combinations of hypergeometric functions 
\[
A(\rho) \, \phi_{\rho,0}(t) - A(\rho') \, \phi_{\rho',0}(t), \quad \rho \ne \rho' \in \sR
\]
continue analytically through $t=1$. Instead of proving this fact directly, we prefer to deduce it as an immediate consequence of our main result. Namely, Theorem~\ref{gamma-apery-thm} will state that, rather generally, for a collection of Frobenius constants $\{ \kappa_{\rho,k}\}_{\rho \in \sR, k \ge 0}$ there exists a meromorphic in $\C$ function whose expansion coefficients at every $s=\rho \in \sR$ coincide with $\kappa_{\rho,k}$ for $k \ge 0$. In the hypergeometric case, if we divide such function by $A(s)^{-1}$ we will get a meromorphic function which is constant near $s = \rho$. Hence the ratio is constant everywhere, and we obtain~\eqref{hg-kappa} simultaneously for all $\rho \in \sR$.   

\begin{proof}[Proof of~\eqref{hg-kappa} for a fixed $\rho \in \sR$.] Local exponents of the operator~\eqref{hg-diff-op} at $t=1$ are given by
\[
0,1,\ldots,r-2, \; \gamma := \sum_{j=1}^r \beta_j - \sum_{j=1}^r \alpha_j - 1
\] 
(see~\cite[(2.8)]{B-H}). First, assume that $\gamma < 0$. In this case for every non-$\sigma_1$-invariant solution $\phi(t)$ of $L$ one has $|\phi(t)| \to \infty$ as $t \to 1$.  Let us show that all the conditions of Lemma~\ref{kappa-lambda-lem} are satisfied. Since the local system of solutions of $L$ is irreducible and $\phi_{\rho,0}$ is an eigenvector of $\sigma_0$, this function is not an eigenvector of $\sigma_1$ and (i) follows. Pochhammer's theorem~\cite[Proposition 2.8]{B-H} implies (ii); for $j > 0$ we note that all $(D-\rho)^j L$ are hypergeometric with the same special local exponent $\gamma$ at $t=1$. Since $\Gamma(s)>0$ when $s>0$, we have $A(s)>0$ when $s \gg 0$ and therefore (iii) holds. (iv) is guaranteed by Lemma~\ref{A-lambda-lem}, which together with Lemma~\ref{kappa-lambda-lem} implies that   
\[
\frac{A(\rho)}{A(s)} = \sum_{n \ge 0} \frac{\kappa_{\rho,n}}{\kappa_{\rho,0}} (s-\rho)^n.
\] 
With the normalization $\kappa_{\rho,0}=A(\rho)^{-1}$ one obtains~\eqref{hg-kappa}, as required.

When $\gamma \ge 0$, we can apply intertwining operators which will decrease $\gamma$ by an integer to reach the above considered case $\gamma < 0$. Namely, let us pick an arbitrary index $1 \le j \le r$. By~\cite[Proposition 2.5]{B-H}, operator $D+\alpha_j$ yields an invertible map from solutions of $L$ to solutions of the order $r$ hypergeometric differential operator $\widetilde L$ with the same set of parameters except for $\alpha_j$ being substituted by $\alpha_j+1$. Note that $\widetilde L$ has the same set of local exponents at $t=0$ (in particular, $\rho$ is a local exponent of $\widetilde L$) and the special local exponent of $\widetilde L$ at $t=1$ equals $\widetilde\gamma = \gamma - 1$. It remains to check that relation~\eqref{hg-kappa} is consistent with the intertwining operator. 

Note that $\widetilde A(s) = (s+\alpha_j) A(s)$, which gives the following relation between the generating functions of Frobenius solutions $\Phi$ and  $\widetilde \Phi$: 
\[\bal
(D+\alpha_j) \Phi(s,t) &= (D+\alpha_j) \sum_{n=0}^\infty a_n(s) t^{n+s} = \sum_{n=0}^\infty (n+s+\alpha_j) a_n(s) t^{n+s} \\
&= (s + \alpha_j) \sum_{n=0}^\infty \widetilde a_n(s) t^{n+s} = (s + \alpha_j) \widetilde \Phi(s,t).
\eal\]
Though we performed this computation symbolically, it clearly remains true is we substitute for $\Phi$ and $\widetilde \Phi$ their formal expansions as power series in $s-\rho$. Suppose that~\eqref{hg-kappa} holds for $\widetilde \Phi$. It means there is a solution $\widetilde \delta (t)$ of $\widetilde L$ such that 
\[
(\sigma_c-1) \widetilde \Phi = \widetilde A(s)^{-1} \widetilde \delta.
\]
Applying $D + \alpha_j$ to the defining relation $(\sigma_c-1) \Phi = \sum_{n=0}^\infty \kappa_{\rho,n} (s-\rho)^n \, \delta(t)$ we find that
\[\bal
\sum_{n=0}^\infty \kappa_{\rho,n} (s-\rho)^n \, (D + \alpha_j)\delta \; &= \; (\sigma_c-1) (D + \alpha_j) \Phi = (s+\alpha_j) (\sigma_c-1) \widetilde \Phi \\
&= \frac{s+\alpha_j}{\widetilde  A(s)} \; \widetilde \delta = \frac1{A(s)} \; \widetilde \delta.
\eal\]
Since the intertwining operator $D+\alpha_j$ commutes with the monodromy operators and solutions $\delta, \widetilde \delta$ span the respective images of $\sigma_c-1$, there is a non-zero constant $c \in \C^\times$ such that $(D+\alpha_j)\delta = c \, \widetilde \delta$. It follows that $\sum_{n=0}^\infty \kappa_{\rho,n} (s-\rho)^n = c A(s)^{-1}$. Renormalizing $\delta(t)$ to $c \delta(t)$, we obtain a collection of Frobenius constants for $L$ satisfying~\eqref{hg-kappa}.
\end{proof}

\begin{remark} One may hope to extend Proposition~\ref{hg-kappa-A-prop} to the case $\alpha_1,\ldots,\alpha_r,\beta_1,\ldots,\beta_r \in \C$ by taking Taylor series expansions for the $\alpha,\beta$ about general points of $\R^{2r} \subset \C^{2r}$. The key point is to show that $\kappa_{\rho,n}$ depend analytically on parameters. For this, one must develop a calculus of deformations $L_R$ of $L$ parametrized by points $R \to \C$ of a $\C$-algebra $R$. One wants the polar locus of $L$ to stay fixed but the local exponents to vary. We have not checked carefully, but a sketch is available under publications at \emph{math.uchicago.edu/\~\,bloch/}
\end{remark}

Hypergeometric connections provide a rare situation in which one can compute all Frobenius constants explicitly. Let us consider a couple of other examples, in which a few of $\kappa_{\rho,n}$'s are computed numerically. These examples motivated our interest in the numbers $\kappa_{\rho,n}$.

\begin{ex}[elliptic curves]\label{beauville-example} A stable family of elliptic curves over $\P^1$ is a family in which the singular fibres have only double points. Such a family has at least four singular points. Stable families with exactly four singular points were classified by Beauville, ~\cite{Beauville}. The differential operator
\[
L = D^2  - t \bigl(11 D^2 + 11 D +  3 \bigr) - t^2 \,  (D+1)^2 \\
\]
is the Picard--Fuchs operator of one of the six families on Beauville's list,~\cite{Z-Apery}. Singularities of $L$ are located at $t = 0, \infty$ and the two roots of $1-11 t - t^2=0$. For each of them there is a single local exponent of double multiplicity, $\rho=0,1,0$ and $0$ respectively. In particular, all four singularities are reflection points.

We used Lemma~\ref{kappa-lambda-lem} and numerical methods described in~\cite{GZ} to compute the values of $ \kappa_n := \kappa_{0,n}$ for the direct path from $0$ to $\frac{-11 + 5 \sqrt{5}}2$. With the precision over 100 digits we recognized a few of them as the following $\Q$-linear combinations of products of zeta values:
\[\bal
&\kappa_0=1,\; \kappa_1=0, \; \kappa_2=-\frac75\zeta(2), \; \kappa_3 = 2 \zeta(3), \; \kappa_4= \frac12\zeta(4), \; \kappa_5=\zeta(5)-3 \zeta(2)\zeta(3), \\
&\kappa_6= \frac{87}{16}\zeta(6)-\frac52\zeta(3)^2, \; \kappa_7=-\frac{55}{8}\zeta(7)-\frac52 \zeta(5)\zeta(2) - \frac54 \zeta(3)\zeta(4).\\
\eal\]
We observed a similar phenomenon for the other families on the list of Beauville.
\end{ex}

\begin{ex}[K3 surfaces]\label{apery-example} The equation $1 - t f(x_1,x_2,x_3)=0$ with
\[
 f(x) = \frac{(x_1-1)(x_2-1)(x_3-1)(1-x_1-x_2+x_1x_2-x_1x_2x_3)}{x_1 x_2 x_3}
\] 
defines a family $X/U$ of K3 surfaces of Picard rank~19 over 
\[
U=\P^1_t \setminus \{ 0, 17 \pm 12 \sqrt{2}, \infty \}.
\] 
The variation $M = H^2(X/U) / NS$ is of rank~3 and there is a class of differential 2-forms $m \in M$ annihilated by the differential operator
\eq{apery-L}{
L = D^3 - t (34 D^3 + 51 D^2 + 27 D + 5) + t^2 (D+1)^3, 
}
see~\cite{K}. The local monodromy of $L$ around $t=0$ is maximally unipotent with the local exponent $\rho=0$. The closest singularity $c=17-12 \sqrt{2}$ is a special reflection point. In~\cite{GZ} Golyshev and Zagier  computed the Frobenius constants $\kappa_n = \kappa_{0,n}$ along with the first higher one for the direct path joining $t=0$ and $t=c$: 
\[
\kappa_0 = 1, \quad \kappa_1=0, \quad \kappa_2 = -\frac{\pi^2}3= -2 \zeta(2), \quad \kappa_3 = \frac{17}{6}\zeta(3).
\]
Note that the objective of~\cite{GZ} is the computation of $\kappa_3$ for the 17 similar families of K3 surfaces, in order to verify the Gamma Conjecture in mirror symmetry. Golyshev and Zagier also evaluate a few more of the higher constants numerically (see~\cite[page 46]{GZ}):
\[\bal
&\kappa_4 = \frac{\pi^4}{45} = \frac45 \zeta(2)^2 = 2 \zeta(4)\\
&\kappa_5 = \frac75 \zeta(5) - \frac{17}3 \zeta(2)\zeta(3)\\
&\ldots\\
&\kappa_{11} = \frac23 \zeta(3,5,3) + \text{ a $\Q$-linear combination of products of} \\
& \qquad\qquad\qquad\qquad\qquad\text{ zeta values of total weight }11
\eal\]
Remarkably, $\kappa_{11}$ is the first one involving multiple zeta values along with ordinary zeta values. 
David Broadhurst was able to find similar experimental expressions in terms of multiple zeta values for the Frobenius constants $\kappa_{n}$ in this example up to $n = 15$,~\cite{Broadhurst}.
\end{ex}
 
Examples~\ref{beauville-example} and~\ref{apery-example} seem to be rather special. We do not expect the Frobenius constants of a Gau\ss--Manin connection to be expressible as $\Q$-linear combinations of products of zeta or multiple zeta values in general. Still, we would like to understand why the numerics is so remarkable in the Beauville families, and other families of Calabi--Yau manifolds that arise in the mirror symmetry program for Fano manifolds. In Section~\ref{sec:Frob-near-MUM} we will show that for a geometric differential operator $L$ (certain expressions in) the Frobenius constants corresponding to actual solutions of $L$ (with $k < m(\rho)$) are periods of the limiting Hodge structure at $t=0$. However, there is no reason to expect that the operators $(D-\rho)^j L$ with $j>0$ are geometric. From this point of view, it is surprising that the higher Frobenius constants in the above examples are periods. This fact will be explained by our main result, which we shall state now.

\begin{thm}\label{gamma-apery-thm} Let $L$ be a differential operator on $U=\P^1 \setminus S$ with regular singularities. Let $K \subseteq \C$ be a field such that there is a monodromy-invariant $K$-structure on the local system of solutions $Sol(L) = Sol(L)_K \otimes_K \C$.  

Assume that $c \ne 0,\infty$ is a special reflection point of $L$ (Definition~\ref{strong-reflection-pt-def}). Let let $\delta(t)$ be a $K$-rational solution of $L$ (a section of $Sol(L)_K$) that spans the image of the variation of the local monodromy around $c$. Let $\{ \kappa_{\rho, n}\}_{\rho \in \sR, n \ge 0}$ be a collection of Frobenius constants for a path $\gamma$ joining $t=0$ and $t=c$, as in Definition~\ref{full-apery-defn}. 

Let $V \subset U_{an}$ be a neighbourhood of $\gamma$ satisfying Assumption~\ref{V-assum}. Then there is a generator $\Gamma_{\xi_0}(s)$ of the $K[e^{\pm 2 \pi i s}]$-module of gamma functions  for the adjoint differential operator $L^\vee$   
\eq{gamma-module-in-main-thm}{
\Gamma_{\xi}(s), \qquad \xi \in H_1(V \,, \, Sol(L^\vee)_K \otimes t^s)
}
such that for every local exponent $\rho \in \sR$ we have
\eq{gamma-kappa-relation-in-main-thm}{
\frac{I(s)}{R(e^{-2 \pi i s})} \Gamma_{\xi_0}(s) = \sum_{n=0}^{\infty} \kappa_{\rho,n} (s-\rho)^n,
}
where $I(s)$ is the indicial polynomial for $L$ at $t=0$ and $R(T) \in K[T]$ is a non-zero polynomial of minimal degree such that $R(\sigma_0)$ annihilates $Image(\sigma_c-1 | Sol(L^\vee))$.
\end{thm} 

Section~\ref{sec:Frob-def-and-gamma} is devoted to the proof of Theorem~\ref{gamma-apery-thm}. Note that the condition of $t=c$ being a special reflection point implies that the $K[e^{\pm 2 \pi i s}]$-module of gamma functions~\eqref{gamma-module-in-main-thm} has rank~1 (see Lemma~\ref{gamma-generators-prop} and the preceding discussion).  Hence a generator of~\eqref{gamma-module-in-main-thm} is defined up to multiplication by $\lambda e^{2 \pi i m s}$ with $\lambda \in K^\times$ and $m \in \Z$. Note that the ambiguity in the definition of Frobenius constants is the same: choosing a different branch of $t^s$ will result in the generating series in the right-hand side of~\eqref{gamma-kappa-relation-in-main-thm} being multiplied by an integer power of $e^{2 \pi i s}$ and rescaling by $\lambda \in K^\times$ corresponds to a different choice of the $K$-rational solution $\delta \in Sol_p(L)$. We see that formula~\eqref{gamma-kappa-relation-in-main-thm} states a correspondence between generators of the module~\eqref{gamma-module-in-main-thm} and collections of Frobenius constants for a given homotopy class of paths~$\gamma$.

\begin{cor}\label{kappas-are-periods} When $L$ is a Picard--Fuchs differential operator, the Frobenius constants $\kappa_{\rho,n}$ belong to the algebra of periods with $2\pi i$ inverted. 
\end{cor}

\begin{proof} Recall from Example~\ref{Picard-Fuchs-ex} that the local system of solutions of a Picard--Fuchs differential operator $L$ is defined over $K=\Q$. First note that if $\phi(t)$ is a period function (a section of $Sol(L)_\Q$) on $U$ and $\sigma$ is a path on $U$ with rational boundary, then $\int_\sigma \phi(t)dt/t$ is a period. Indeed, possibly shrinking $U$ and deforming $\sigma$, we can find a smooth $f: X\to U$, a continuously varying family  
$\{\gamma_t \subset X_t\}$ 
of closed topological chains, and a class $\omega \in H^*_{DR}(X/U)$, such that 
we can realize $\phi(t) = \int_{\gamma_t}\omega_t$. Because $dt/t$ is a closed $\C$-analytic form of top dimension on $U$, $\wedge dt/t$ defines a map $\Omega^*_{X/U} \to \Omega^{*+1}_X$. We have by Fubini:
$$\int_\sigma \phi(t)dt/t = \int_S (\omega_t \wedge dt/t);\quad S = \bigcup_{t\in \sigma}\gamma_t.
$$
Note that the $\gamma_t$ are closed chains, so the boundary of $S$ is contained in $f^{-1}(\text{boundary of }\sigma)$. It follows that $\int_\sigma \phi(t)dt/t$  is a period in the sense of~\cite{KZ-Periods}. 

Finally, we want to show $\int_\sigma \phi(t)\log(t)^ndt/t$ is a period, assuming the integrand is single-valued along $\sigma$. Let $\tau_t: [0,1] \to \P^1$ be a continuous (in $t$) family of paths from $1$ to $t$. Then $\Sigma:=\bigcup_{t\in \sigma}\{\tau_t[0,1]\}^n$ defines a chain in $U\times (\P^1)^n$. We have
$$\int_\Sigma \phi(t)\wedge (dx/x)^n\wedge dt/t = \int_\sigma \phi(t)\log(t)^n dt/t.
$$
Here the lefthand side is a period, again by Fubini. 

We apply this to formula~\eqref{gamma-kappa-relation-in-main-thm} above. Since $L$ is a Picard--Fuchs differential operator, its local exponents $\rho$ are rational numbers (see e.g.~\cite{Katz70}). Differentiate both sides $k$ times with respect to $s$ and evaluate at $s=\rho$. On the lefthand side the result will be a $\overline{\Q}$-linear combination of products of integral powers of $2 \pi i$ and integrals of the form $\int_\xi \phi(t)t^{\rho} \log (t)^k dt/t$ with $k \le n$ and $\phi \in Sol(L^\vee)_\Q$. Since $\rho \in \Q$, the product $t^{\rho} \phi(t)$ is again a period function, and the above argument shows that values of such integrals are periods.
\end{proof}

\begin{remark}[Ap\'ery's constants] The differential operator $L$ from Example~\ref{apery-example} is closely related to the famous proof of irrationality of $\zeta(3)$,~\cite{Apery79}. Namely, Ap\'ery shows that the two sequences 
\begin{small}\[\bal
&A_n = \sum_{k=0}^n \binom{n}{k}^2 \binom{n+k}{k}^2 \in \Z,\\
&B_n = \sum_{k=0}^n \binom{n}{k}^2 \binom{n+k}{k}^2 \left( \sum_{m=0}^n \frac1{m^3} + \sum_{m=1}^k \frac{(-1)^{m-1}}{2 m^3 \binom{n}{m} \binom{n+m}{m} } \right) \in \Q
\eal\]
\end{small}
satisfy 
\eq{Apery-approx}{
|B_n - A_n \zeta(3)| < c^n,
}
where $c=17-12\sqrt{2}=0.029\ldots$ is the closest to $0$ reflection points of $L$. Irrationality then follows from a bound for denominators of $\{ B_n \}$. The above sequences satisfy a recurrence relation, which can be stated in terms of their generating functions $A(t):=\sum_{n \ge 0} A_n t^n$ and $B(t):=\sum_{n \ge 1} B_n t^n$  as
$L A = 0$, $L B = 6 t$. In order to include $B$ we need to consider the operator $\widetilde L = (D -1) L$, for which $c=17-12\sqrt{2}$ is again a special reflection point. In terms of the Frobenius solutions, we then have $A(t) = \phi_{0,0}(t)$, $B(t) = 6 \, \phi_{1,0}(t)$. Ap\'ery's approximation~\eqref{Apery-approx} shows that solution $B(t)-\zeta(3)A(t)$ extends analytically through $t=c$ or, equivalently, that $\frac{\kappa_{1,0}}{\kappa_{0,0}} = \frac{\zeta(3)}{6}$. Relation between this number  and the Frobenius constants $\kappa_{0,m}$, $m \ge 0$ from Example~\ref{apery-example} is as follows. 

One of the outcomes of Theorem~\ref{gamma-apery-thm} is existence of a meromorphic function 
\[
\kappa(s) := I(s)R(e^{-2 \pi i s})^{-1}\Gamma_{\xi_0}(s)
\]
in the whole complex plane $\C$ whose expansion coefficients at each $\rho \in \sR$ are equal to $\kappa_{\rho,0}, \kappa_{\rho,1}, \ldots$ For the operator $L$ from Example~\ref{apery-example} and normalization $\kappa_{0,0}=1$ we have
\[
\kappa(s) = 1 - \tfrac{\pi^2}{3} s^2 + \tfrac{17}{6}\zeta(3)s^3 + \ldots
\] 
The reader may check that the Frobenius method for $\widetilde L = (D -1) L$ yields the same sequence of rational functions $\{ a_n(s); n \ge 0\}$ as in the Frobenius method  for $L$. Therefore we get the same function $\kappa(s)$ as above and, according to the previous paragraph, one has 
\[
\kappa(1) = \tfrac{\zeta(3)}{6}.
\] 
There seem to be no direct link between this value and appearance of $\zeta(3)$ in $\kappa_{0,3}=\frac{17}6\zeta(3)$. (We refer to a forthcoming paper of Kerr, \cite{K2} for more details.)

Gamma functions satisfy difference equations (Proposition~\ref{functional_equation}), and the discussion just before Proposition~\ref{functional_equation} shows that $\kappa(s)$ satisfies a difference equation as well. In our example the functional equation has three terms, and we can conclude that
$\kappa(m) \in \Q \zeta(3) + \Q$ for any $m \in \Z$.        
\end{remark}

To get a sense of what the polynomial $R(T)$ in~\eqref{gamma-kappa-relation-in-main-thm} looks like, let us make a few observations. For simplicity, let us assume that $I(s)$ has no pair of roots that differ by a non-zero integer. Then $\sR$ is the full set of roots of $I(s)$ and the characteristic polynomial of $\sigma_0$ on $Sol_p(L)$ is given by $\prod_{\rho \in \sR}(T-e^{2 \pi i \rho})^{m(\rho)}$, see Remark~\ref{local-exponents-and-monodromy-rmk}. It is easy to see that the indicial polynomial of $L^\vee$ at $t=0$ is given by $I(-s)$, and hence the characteristic polynomial of $\sigma_0$ on $Sol_p(L^\vee)$ is given by $P(T)=\prod_{\rho \in \sR}(T-e^{-2 \pi i \rho})^{m(\rho)}$. Since  $P(\sigma_0)$ vanishes on $Sol_p(L^\vee)$, it follows that $R(T)$ divides $P(T)$. We will end this section with making the statement of Theorem~\ref{gamma-apery-thm} more concrete in a special case when $\sigma_0$ is \emph{maximally unipotent}. 

\begin{cor}\label{MUM-case-of-main-thm} Let the assumptions be as in Theorem~\ref{gamma-apery-thm} and suppose that the indicial polynomial of $L$ at $t=0$ is $I(s)=s^r$. Then generators $\Gamma_{\xi_0}(s)$ of the $K[e^{\pm 2 \pi i s}]$-module of gamma functions~\eqref{gamma-module-in-main-thm} satisfy
\[
\Gamma_{\xi_0}(0) \ne 0,
\]
and the Frobenius constants $\kappa_n := \kappa_{0,n}$ are given by
\[
\frac{s^r}{(1-e^{-2 \pi i s})^{r-d}} \Gamma_{\xi_0}(s) = \sum_{n=d}^\infty \kappa_n s^n,
\] 
where $d := \min\{n \ge 0 : \kappa_n \ne 0\} < r$.
\end{cor}    
\begin{proof} Let us first show that $d < r$. In our case $I(s)$ has only one root $\rho=0$, and hence the Frobenius solutions $\phi_k := \phi_{0,k}$ with $0 \le k < r$ form a basis in $Sol_p(L)$. It follows that their images under $\sigma_c-1$ span $Image(\sigma_c-1 | Sol_p(L))$. In particular, at least one of the first $r$ Frobenius constants $\kappa_0,\ldots,\kappa_{r-1}$ is non-zero. 

Since the characteristic polynomial of $\sigma_0$ on $Sol_p(L^\vee)$ is given by $P(T)=(T-1)^r$ and $R(T)$ in Theorem~\ref{gamma-apery-thm} divides $P(T)$, we can simply take $R(T)=(1-T)^m$ where the degree $m \le r$ is determined as follows. With any non-zero $\delta^\vee \in Image( \sigma_c-1|Sol_p(L^{\vee}))$ we have 
\[
m = \min \{ k > 0 : \delta^\vee, \sigma_0 \delta^\vee, \ldots, \sigma_0^k \delta^\vee \text{ are linearly dependent} \}.
\]
Relation~\eqref{gamma-kappa-relation-in-main-thm} now reads as
\eq{kappa-gamma-rel-MUM}{
\frac{s^r}{(1 - e^{-2 \pi i s})^m} \Gamma_{\xi_0}(s) = \sum_{n=d}^\infty \kappa_n s^n, \quad \kappa_d \ne 0.
}
It remains to show that $m+d=r$. Since $\Gamma_{\xi_0}(s)$ is an entire function, it follows from~\eqref{kappa-gamma-rel-MUM} that $m+d \ge r$. To show that $m+d \le r$, note that  $\{ \delta^\vee, \phi_{k}\} = 0 $ for $0 \le k < d$ because the first $d-1$ Frobenius solutions are $\sigma_c$-invariant. Since $\sigma_0$ preserves the $d$-dimensional subspace $A := \sum_{k=0}^{d-1} \C \phi_k \subset Sol_p(L)$, it follows that $A$ is orthogonal to the $m$-dimensional subspace $\C[\sigma_0] \delta^\vee$. Since the duality pairing $\{*,*\}$ is perfect, it follows that $m + d \le r$.
\end{proof}

\section{Proof of the relation between Frobenius constants and gamma functions}\label{sec:Frob-def-and-gamma}

In this section we prove Theorem~\ref{gamma-apery-thm}. We start with some preparations. Let $M$ be an algebraic connection of rank $r$ on an open subset $U \subseteq \G_m$ and $\sO=\sO_U$ be the ring of functions of $U$. The ring of differential operators $\sD = \sD_U$ is generated over $\sO$ by the derivation $D=t \frac{d}{dt}$. We fix a generator $m \in M$ and a differential operator
\[
L = q_0(t) D^r + q_1(t) D^{r-1} + \ldots + q_r(t), \quad q_i \in \sO
\]
such that $L m = 0$. Possibly after passing to a smaller open set $U$, we assume that $q_0 \in \sO^{\times}$ is a unit. It follows that $M \cong \sD / \sD L$ is a free $\sO$-module of rank $r$ with the basis $m,Dm,\ldots,D^{r-1}m$. 

The dual connection on $M^\vee = Hom_\sO(M,\sO)$ is determined by the identity
\eq{dual-connection}{
\la \xi, D \eta \ra + \la D \xi, \eta \ra = D \la \xi, \eta \ra , \qquad \xi \in M, \eta \in M^\vee.
}
Let $e_0,\ldots,e_{r-1} \in M^\vee$ be the basis dual to $m,Dm,\ldots,D^{r-1}m \in M$, that is we have
\[
\la D^j m, e_i \ra = \delta_{i,j}.
\]

Using~\eqref{dual-connection} one can easily check that the $\sD$-module structure on $M^\vee$ is given by
\eq{dual-conn}{D e_i + e_{i-1} = \frac{q_{r-i}}{q_0} e_{r-1},\quad 0 \le i \le r-1,
}
with the convention that $e_{-1}=0$. The adjoint differential operator is defined as
\[
L^\vee = (-D)^r q_0(t) + (-D)^{r-1} q_1(t) +\ldots + q_r(t).
\]
Below we will check that $M^\vee \cong \sD/\sD L^\vee$. An important role in the proof of Theorem~\ref{gamma-apery-thm} is played by the following $\C$-linear bracket 
\eq{bracket}{
\bal
&\{*,*\} \;:\; \sO^{an} \otimes_\C \sO^{an} \to \sO^{an},\\
&\{\psi,\phi\} = \underset{\begin{tiny}\bal &h,\nu,i \ge 0 \\ h+\nu&+i=r-1 \eal\end{tiny}} \sum (-D)^h (q_\nu \psi) D^i \phi.
\eal}
We will show that this bracket satisfies
\eq{int-by-parts}{
D\{\psi,\phi\} = \psi(L \phi) - (L^\vee\psi)\phi \,.
}
In particular, $\{*,*\}$ is constant on  $Sol(L^\vee) \otimes_{\C} Sol(L)$ and coincides with the duality pairing $\langle *, * \rangle$ on $\sM \otimes_K \sM^\vee$. It seems that property~\eqref{int-by-parts} goes back at least to Lagrange. For the sake of completeness, we will give its proof in the next series of lemmas. The reader willing to take~\eqref{int-by-parts} for granted may proceed directly to the proof of Theorem~\ref{gamma-apery-thm}.

}

\begin{lem}\label{dual-via-adjoint-L-lemma}We have $L^\vee (q_0^{-1} e_{r-1}) = 0$. Thus, the map $1\mapsto q_0^{-1} e_{r-1}$ yields an isomorphism $\sD/\sD L^\vee \cong M^\vee$. 
\end{lem}
\begin{proof} Using~\eqref{dual-conn} we compute
\[L^\vee (q_0^{-1} e_{r-1}) = \sum_{i=0}^r (-D)^i \left(\frac{q_{r-i}}{q_0}e_{r-1} \right)= \\
\sum_{i=0}^{r-1}(-D)^i(D e_i +e_{i-1}) + (-D)^r e_{r-1}   = 0.
\]

\end{proof}

\begin{lem}\label{m-vee-rho-lemma} Let $m^\vee := q_0^{-1} e_{r-1} \in M^\vee$ be the element from Lemma~\ref{dual-via-adjoint-L-lemma}. Let $\rho_0,\ldots,\rho_{r-1} \in M$ be the basis dual to $m^\vee, Dm^\vee,\ldots,D^{r-1} m^\vee \in M^\vee$.  We have
\eq{pairing-e-rho}{\langle \rho_j, e_i \rangle = \underset{\begin{tiny}\bal &h,\nu \ge 0 \\ h+\nu&=r-1-i \eal\end{tiny}} \sum (-1)^h \binom{h}{j} D^{h-j}q_\nu
}
for all for $0\le j, i\le r-1$. In particular,  $\langle \rho_j, e_i \rangle = 0$ when $j+i>r-1$ and we have $\rho_{r-1} = (-1)^{r-1} q_0 m$.
\end{lem}
\begin{proof} 
Applying~\eqref{dual-conn} consequently with $i=r-1,r-2, \ldots$ we find that 
\[\bal
e_{r-1-k} &= (-D)^k e_{r-1}+(-D)^{k-1}\left(\frac{q_1}{q_0}e_{r-1}\right)+\cdots + (-D)\left(\frac{q_{k-1}}{q_0}e_{r-1}\right) + \frac{q_k}{q_0}e_{r-1} \\
&= \sum_{\nu=0}^{k} (-D)^{k-\nu}(q_\nu m^\vee)
\eal\]
for all $0\le k \le r-1$.
The coefficient of $D^j m^\vee$ in the latter sum equals $0$ when $j>k$ and
\[
\sum_{\nu=0}^{k-j}(-1)^{k-\nu}\binom{k-\nu}{j}D^{k-\nu-j}q_\nu 
\]
when $j \le k$. This expression for $k=r-1-i$ coincides with~\eqref{pairing-e-rho}.

\end{proof}

One can easily check that an element $\sum_{i=0}^{r-1} \phi_i e_i  \in M_{an}^\vee$ is a horizontal section if and only if $\phi_i = D^i \phi_0$ for all $i>0$ and $L \phi_0 = 0$. This observation identifies the local system $Sol(L)$ with $\sM^\vee = (M_{an}^\vee)^{\nabla^\vee=0}$.

\begin{lem} Recall that $\{e_i\}\subset M^\vee$ is the dual basis to $\{D^im\}\subset M$. Similarly, $\{\rho_i\} \subset M$ is the dual basis to $\{D^i m^\vee \}\subset M^\vee$. With this notation, define $\C$-linear maps $\eta: \sO^{an} \to M_{an}^\vee$ (resp. $\xi: \sO^{an} \to M_{an}$) by 
\eq{eta-phi-maps}{\eta(\phi) = \sum_{i=0}^{r-1} (D^i\phi)e_i;\quad \xi(\psi) = \sum_{i=0}^{r-1} (D^i\psi)\rho_i.
} 
We define a bracket $\{*,*\} : \sO^{an} \otimes_\C \sO^{an} \to \sO^{an}$ by composing the tensor product
\[\bal& \; \xi \otimes \eta : \sO^{an} \otimes_\C \sO^{an} \to M_{an} \otimes_{\sO^{an}} M_{an}^\vee; \\
 &(\xi \otimes \eta)(\psi\otimes \phi) = \sum_{j,i=0}^{r-1}(D^j\psi)(D^i\phi)\rho_j\otimes e_i
\eal\]
with the duality pairing $\langle * , * \rangle$:
\[\{\psi,\phi\} \; := \langle \xi(\psi), \eta(\phi) \rangle = \sum_{j,i=0}^{r-1} (D^j\psi) (D^i\phi)\langle \rho_j,e_i\rangle. 
\]
This bracket satisfies
\eq{int-by-parts1}{
D  \{\psi,\phi\}  = \psi (L \phi) - (L^\vee \psi) \phi 
}
and can be expressed as
\eq{bracket1}{
\bal \{\psi,\phi\} = \underset{\begin{tiny}\bal h,&\nu,i \ge 0 \\ h+\nu&+i=r-1 \eal\end{tiny}} \sum (-1)^h \Bigl( D^h(q_\nu \psi) \Bigr) D^i \phi. \\
\eal}
\end{lem}
\begin{proof}To prove~\eqref{int-by-parts1}, we observe that 
\[\bal
D (\eta(\phi)) &= \sum_{i=0}^{r-1} \left(  (D^{i+1}\phi) e_i +  (D^i \phi) D e_i \right)  \\
&= \sum_{i=0}^{r-1} (D^{i+1}\phi) e_i +  \sum_{i=0}^{r-1}(D^i \phi)(q_{r-i} m^\vee - e_{i-1})  = (L \phi) m^\vee,
\eal\]
where we used the convention $e_{-1}=0$ in the summation in the last line. The same computation for $L^\vee$ would yield
\[
D (\xi(\psi)) = -(L^\vee \psi) m.
\]
Indeed, since $\rho_{r-1}=(-1)^{r-1} q_0 m$ (see Lemma~\ref{m-vee-rho-lemma}) and the leading coefficients of $L^\vee$ is $(-1)^r q_0$, we obtain $(m^\vee)^\vee = -m$. Finally, using~\eqref{dual-connection} we obtain
\[\bal
D \{ \psi, \phi \} &= D \langle \xi(\psi), \eta(\phi) \rangle = \langle D\xi(\psi), \eta(\phi) \rangle + \langle \xi(\psi), D\eta(\phi) \rangle \\
&= - (L^\vee \psi) \langle m, \eta(\phi) \rangle + (L \phi) \langle \xi(\psi), m^\vee \rangle = - (L^\vee \psi) \phi + (L \phi) \psi.
\eal\] 
Expression~\eqref{bracket1} follows from Lemma~\ref{m-vee-rho-lemma} by a straightforward computation: expanding \eqref{bracket1}, we see that the coefficient at $D^j(\psi) D^i(\phi)$ coincides with~\eqref{pairing-e-rho}. 
\end{proof}

\begin{proof}[Proof of Theorem~\ref{gamma-apery-thm}] Fix a local exponent $\rho \in \sR$. We will work with formal power series in $(s-\rho)$ whose coefficients are analytic functions in a neighbourhood of a regular base point $p$ chosen on the path $\gamma$.  By the construction of the Frobenius solutions in Section~\ref{sec:Apery} it follows that the formal series $\Phi = \sum_{n \ge 0} \phi_{\rho,n}(t) (s-\rho)^n$ satisfies
\eq{Frob-de}{
L \Phi = I(s) t^s,
}
\eq{Frob-around-0}{
\sigma_0 \Phi = e^{2 \pi i s} \Phi
}
and
\eq{Frob-around-c}{
(\sigma_c-1) \Phi = \kappa(s) \, \delta(t),
}
where $\kappa(s)$ denotes the formal series $\sum_{n \ge 0} \kappa_{\rho,n} (s-\rho)^n$ and $I(s) t^s$  in~\eqref{Frob-de} and $e^{2 \pi i s}$ in~\eqref{Frob-around-0} respectively have to be expanded into power series at $s=\rho$. 

For any class $\xi \in H_1(V_\gamma \,, Sol(L^\vee) \otimes t^s)$ represented by a 1-cycle $\sum_j \sigma_j \otimes \psi_j \otimes e^{2 \pi i s n_j}$ we compute the respective gamma function using~\eqref{Frob-de} as follows:
\eq{int-by-parts-1}{\bal
I(s) \Gamma_\xi(s) &= \sum_j e^{2 \pi i s n_j} I(s) \int_{\sigma_j} t^s \psi_j(t) \frac{dt}{t} = \sum_j e^{2 \pi i s n_j} \int_{\sigma_j} (L \Phi)(t) \psi_j(t) \frac{dt}{t}  \\
& = \sum_j e^{2 \pi i s n_j} \int_{\sigma_j} d \{ \psi_j, \Phi \}  = \sum_j e^{2 \pi i s n_j} (\sigma_j-1) \{ \psi_j, \Phi \}.   
\eal}
Here to pass to the second line we used the fact that $L^\vee \psi_j = 0$ and hence $(L \Phi)(t) \psi_j(t) = (L \Phi)(t) \psi_j(t) - \Phi(t) (L^\vee \psi_j)(t) = D \{ \psi_j, \Phi \}$ due to property~\eqref{int-by-parts} of the bracket $\{ *, *\}$. As above, all functions of $s$ in our formulas mean the respective expansions into power series at $s=\rho$.  

According to the statement of Theorem~\ref{gamma-apery-thm}, $R(\sigma_0)=\sum_m \lambda_m \sigma_0^m \in K[\sigma_0^{\pm 1}]$ generates the annihilator of $\delta^\vee$ in this ring. By Lemma \ref{annlem}, the module of gamma functions~\eqref{gamma-module-in-main-thm} is generated by $\Gamma_{\xi_R}(s)$ where $\xi_R$ is the class of the 1-cycle $\sum_m \lambda_m \sigma_0^m \otimes \delta^\vee \otimes e^{-2 \pi i s m}  + \sigma_c \otimes \ve^\vee \otimes R(e^{- 2 \pi i s})$, where $\ve^\vee$ is any solution of $L^\vee$ satisfying $(\sigma_c-1) \ve^\vee = \delta^\vee$. We assume that $\delta$, $\delta^\vee$ and $\ve^\vee$ are $K$-rational solutions. Applying~\eqref{int-by-parts-1} to the class $\xi_R$, we compute further using~\eqref{Frob-around-0} and~\eqref{Frob-around-c}:
\[\bal
I(s) \Gamma_{\xi_R}(s) &= \sum_m \lambda_m e^{-2 \pi i s m} (\sigma_0^{m}-1) \{\delta^\vee, \Phi \} + R(e^{-2 \pi i s}) (\sigma_c-1) \{\ve^\vee, \Phi\}\\
&= \sum_m \lambda_m e^{-2 \pi i s m} \bigl( \{ \sigma_0^{m} \delta^\vee,  e^{2 \pi i s m} \Phi\} - \{\delta^\vee, \Phi \}  \bigr)\\
& \qquad\qquad + R(e^{-2 \pi i s})\bigl(\{\sigma_c \ve^\vee, \sigma_c \Phi \} - \{\ve^\vee, \Phi\}\bigr)\\
&= \{ R(\sigma_0) \delta^\vee, \Phi\} - R(e^{-2 \pi i s}) \{\delta^\vee, \Phi\} \\
& \qquad\qquad + R(e^{-2 \pi i s}) \bigl(\{ \ve^\vee + \delta^\vee , \Phi + \kappa(s) \delta\} - \{\ve^\vee, \Phi\}\bigr)\\
& = R(e^{-2 \pi i s}) \, \kappa(s) \, \{ \ve^\vee+\delta^\vee, \delta\}.
\eal\]
Since  $\delta$ and $\ve^\vee + \delta^\vee$ are $K$-rational, we have $\{ \ve^\vee+\delta^\vee, \delta\}\in K$. To show that this number is non-zero, notice that $\{ \psi, \delta \} = 0$ if and only if $\sigma_c \psi = \psi$. (This follows from the remark below formula \eqref{int-by-parts} that the bracket coincides with the duality pairing on solutions $\sM\times \sM^\vee$, so the orthogonal to the image of $\sigma_c-1$ on $\sM$ is the kernel of $\sigma_c-1$ on $\sM^\vee$.) One can easily check that 
\[
(\sigma_c-1)(\ve^\vee + \delta^\vee) = \delta^\vee + (\sigma_c-1) \delta^\vee = \sigma_c \delta^\vee \ne 0. 
\]
Note that the class $\xi_0 = \{ \ve^\vee+\delta^\vee, \delta\}^{-1} \xi_R \in H_1(V_\gamma \, , \, Sol(L^\vee)_K \otimes t^s )$ does not depend on our choice of $\rho \in \sR$. The above computation shows that $I(s) \Gamma_{\xi_0}(s) /R(e^{- 2 \pi i s}) = \kappa(s)$, which completes the proof.  
\end{proof}

\begin{cor}Let notation be as in the theorem, and let $\psi$ be a solution for $L^\vee$. Then the bracket $\{\psi,\Phi\}$ is a potential for the $\Gamma$-integrand $\psi(t)t^sdt/t$, i.e. $d\{\psi,\Phi\}=\psi t^sdt/t$. This is a special property of Mellin transforms for solutions of $\sD$-modules.
\end{cor}

\section{The limiting MHS and Frobenius constants}\label{sec:Frob-near-MUM}

We are given a Zariski open $U \subset \G_m \subset \P^1_{t}$ and a regular singular point connection $\nabla: M \to M\otimes \Omega^1_U$ defined over a subfield $K \subseteq \C$. We assume that $(M,\nabla)$ is a direct summand in the Gau\ss--Manin connection on $H^{w}_{DR}(X/U)$ for some projective, smooth $X \to U$. In this case $M$ carries a variation of pure Hodge structure of weight $w$. 

As earlier, we denote by $\sD$ be the ring of differential operators on $U$. This ring is generated over the ring of regular functions $\sO_U$ by the derivation $D=t\frac{d}{dt}$. We assume there is an element in the lowest Hodge submodule, $m \in F^{w}M$, and a differential operator 
\[
L = \sum_{j=0}^r q_j(t) D^{r-j} \in \sD, \quad q_0 \in \sO_U^\times
\]
such that $Lm=0$ and
\[
M\cong \sD/\sD L,
\]
the isomorphism being given by $m \mapsto 1$.

We further assume that $t=c$ is a reflection point for $M$ (Definition~\ref{c-assum}) and that $t=0$ is a point of maximally unipotent monodromy (\emph{MUM}). The latter condition means that in a certain basis of solutions of $M$ the operator of local monodromy around $t=0$ is given by  
$\sigma_0=\exp(N)$, 
\eq{N-matrix}{
N=\begin{pmatrix}0 & 1 & 0 & \ldots \\ 0 & 0 & 1 & \ldots \\ & \ldots \end{pmatrix}.
}
In this case $L$ has a unique local exponent $\rho \in \Z$ at $t=0$. By the construction given is Section~\ref{sec:Apery}, for every homotopy class of paths $\gamma$ joining $t=0$ and $t=c$ there is an infinite collection of Frobenius constants $\kappa_{n} := \kappa_{\rho,n}$, $n=0,1,\ldots$ (see Definition~\ref{full-apery-defn}; possibly dividing $L$ on the left by a power of $(t-c) \in \sO^\times_U$, we can assure that $t=c$ is a special reflection point for this differential operator).  In this section we will show that certain expressions in the Frobenius constants $\kappa_0,\ldots,\kappa_{r-1}$ and $2 \pi i$ are periods of the limiting mixed Hodge structure at $0 \in \P^1 \setminus U$.  In order to include the higher Frobenius constants $\kappa_r, \kappa_{r+1}, \ldots$ in the picture, we start by building a variation of mixed Hodge structure on extensions of $M$ by powers of the Kummer connection.

We are interested in the analytic structure of $M$ in a punctured disk $\Delta^*$ about $0\in \P^1 \setminus U$. The following is classic.
\begin{thm}\label{odp} The evident functor is an equivalence of categories between the category of analytic connections on $\Delta^*$ meromorphic at $0$ and having at worst regular singular points there and 
the category of all analytic connections on $\Delta^*$.\end{thm}
\begin{proof}See for example \cite{M}, th\'eor\`eme (1.1), p. 24. 
\end{proof}
Of course, the category of analytic connections is equivalent to the category of local systems, so, for example, $M_{\Delta^*}$ is determined by its monodromy at $0$ which is $\sigma_0=\exp(N)$ with $N$ as in \eqref{N-matrix}. We write $\sO$ for the ring of analytic functions on $\Delta^*$ which are meromorphic at $0$. For an integer $n\ge 0$, define a free $\sO$-module with basis $\{e_i\}$
$$\E_n := \sO e_{-n-1}\oplus \sO e_{-n}\oplus \cdots \sO e_{0}\oplus \sO e_{1}\oplus \cdots \sO e_{r-1}.
$$
We define a connection on $\E_n$
$$\nabla e_i := e_{i-1}\frac{dt}{ t};\quad \nabla(e_{-n-1})=0. 
$$
Let $\F_n \subset \E_n$ be the submodule spanned by $e_{-1},\dotsc,e_{-n-1}$. The connection restricts to a connection on $\F_n$, so there is an induced connection on $\E_n/\F_n$. (Note $\F_0=\sO e_{-1}$ with trivial connection.) 

\begin{defn} The Kummer connection $K_t = \sO\ve \oplus \sO\eta$ is the rank $2$ connection  with $\nabla(\ve) = \eta \frac{ dt}{ t}$ and $\nabla(\eta)=0$.  
\end{defn}
\begin{lem}We have an isomorphism of connections
$$\F_n \cong \text{Sym}^n K_t. 
$$
(Here $\text{Sym}^0K_t = \sO$ with the trivial connection.) 
\end{lem}
\begin{proof}The identification is given by
$$e_{-1}=\ve^n/n!,\ e_{-2} = \ve^{n-1}\eta/(n-1)!,\dotsc,e_{-n-1}=\eta^n.
$$
\end{proof}

The dual connection on the free $\sO$-module $\E_n^\vee$  with basis $e_i^\vee,\ -n-1\le i\le r-1$ is given by
$$\nabla^\vee e_i^\vee = -e^\vee_{i+1}\frac{dt}{t}.
$$
Consider horizontal sections of $\E^\vee_n$, (that is, solutions for $\E_n$) given by
\eq{horsol-eps}{\ve_k(t) = \frac{\log(t)^k}{k!}e^\vee_{r-1}+\frac{\log(t)^{k-1}}{(k-1)!}e^\vee_{r-2}+\cdots + \log(t)e^\vee_{r-k}+e^\vee_{r-k-1}
}
for $k=0,\ldots,n+r$. Note that in the basis $\rho_k := (2 \pi i)^{-k} \ve_k$ operator $N = \log(\sigma_0)$ acts by the nilpotent matrix~\eqref{N-matrix} of size $n+r+1$.  

\begin{lem} Continuing to assume $0$ is a MUM point, we get an exact sequence of meromorphic connections with regular singular points at the origin
$$0 \to \text{Sym}^n K_t \to \E_n \to M_{\Delta^*} \to 0.
$$
\end{lem}
\begin{proof} The assumption guarantees that the monodromy for $\E_n/\F_n$ coincides with that of $M_{\Delta^*}$, so one can apply Theorem \ref{odp}. 
\end{proof}

We now define a variation of Hodge structure on $\E_n$. Recall that $M$ is a pure variation of Hodge structure on $U$, with Hodge filtration $0\subset F^{w}M\subset F^{w-1}M\subset\cdots \subset F^0M=M$.  We are given an element $m\in F^{w}M$ satisfying a differential equation $Lm=0$ and such that $M \cong \sD/\sD L$. Without loss of generality, we assume from now on that $q_0(0) \ne 0$ and the unique local exponent of $L$ at $t=0$ is 
\[
\rho=0.
\]
Under our assumption on $\sigma_0$, the latter condition can be always achieved by multiplying $m$ by a power of $t \in \sO_U^\times$. In this case the monodromy of solutions of the differential operator $D^{n+1} L$ around $t=0$ is again maximally unipotent (this follows from the existence of higher Frobenius solutions, \eqref{Frob-functions}, \eqref{DL}) and hence it is the same as the monodromy of the local system of horizontal sections $\sE_n^\vee = \E_n^{\vee,\nabla^\vee=0}$.  We can again apply Theorem \ref{odp} and build a commutative diagram of connections \minCDarrowwidth.1cm
\eq{connect}{\begin{CD}0 @>>> (\sD/\sD D^{n+1})_{\Delta^*} @>>R_L >  (\sD/\sD D^{n+1}L)_{\Delta^*} @>>> (\sD/\sD L)_{\Delta^*} @>>> 0 \\
@. @VV \cong V   @VV \cong V    @VV \cong V    \\
0 @>>> \text{Sym}^n K_t @>>> \E_n @>>> M_{\Delta^*} @>>> 0.
\end{CD}
}
(Here $R_L$ means right multiplication by $L$.) The Frobenius solutions $\phi_k := \phi_{0,k}$ for $D^{n+1}L$ are given by
\[
\phi_k(t) = \sum_{j=0}^k  \frac{\log(t)^j}{j!} \phi_{k-j}^{an}(t), \quad 0 \le k \le n+r,
\]
where $\phi^{an}_j \in \sO$ are uniquely defined analytic functions (see~\eqref{Frob-functions}) satisfying the condition
\eq{normalization}{\phi^{an}_0(0)=1;\quad \phi^{an}_j(0)=0,\ j>0.} 
These solutions satisfy
\eq{Lphi}{\bal
& L(\phi_0)=\cdots = L(\phi_{r-1})=0,\\
& L\phi_{r+j} = \frac{\log(t)^{j}}{j!};\quad 0\le j\le n.
\eal}
Consider the element 
\eq{phi}{ \phi^{an} := \sum_{j=0}^{n+r} \phi^{an}_j(t) e_{r-1-j} \in \E_n.
}
Applying the solutions $\ve_k$ defined in~\eqref{horsol-eps}, we observe that 
\eq{Frob-basis-via-phi-an}{
\phi_k = \langle \phi^{an}, \ve_k\rangle
}
for each $0 \le k \le n+r$. It then follows that $\phi^{an} \in \E_n$ satisfies $D^{n+1}L \phi^{an} = 0$. In particular, one can assume that the vertical isomorphism in the diagram~\eqref{connect} is given by $1 \in \sD \mapsto \phi^{an} \in \E_n$ and that in the bottom row  $\phi^{an}$ maps to the element $m \in M_{\Delta^*} \cong \E_n/\text{Sym}^n K_t$.

We can use $\phi^{an}$ to define a variation of mixed Hodge structure on $\E_n$ as follows. Since $m, Dm, D^2m,\dotsc,D^{r-1}m$ form an $\sO$-basis for $M_{\Delta^*}$, mapping $D^im \mapsto D^i\phi^{an}$ defines an $\sO$-module splitting $s: M_{\Delta^*} \to \E_n$, i.e. 
\eq{split}{\E_n \cong \text{Sym}^nK_t \oplus s(M_{\Delta^*}).
} 
One knows that $K_t$ is the connection underlying a variation of mixed Tate Hodge structure (Kummer variation). Its $n$th symmetric power is the variation on $\text{Sym}^n K_t$ where the Hodge filtration is given by
\[
F^0 \subset F^{-1} \subset \ldots \subset F^{-n}=\text{Sym}^n K_t, \quad F^{-k} = \oplus_{j=0}^{k} \sO \ve^{n-j}\eta^j
\]
and the weight filtration is given by
\[\bal
&W_{-2n} = W_{1-2n} \subset W_{2-2n} = W_{3-2n} \subset \ldots \subset W_0=\text{Sym}^n K_t, \\
&W_{-2k} = \oplus_{j=k}^{n} \sO \ve^{n-j}\eta^j.
\eal\]
The Hodge filtration is opposite to the weight filtration in the sense that $F^{-k}\oplus W_{-2(k+1)}=\text{Sym}^n K_t$, from which one can easily compute that $gr^W_{-2k} \text{Sym}^n K_t \cong \Q(k)$. The corresponding weight-graded Hodge variation is thus given by $\Q(n)\oplus  \ldots \oplus \Q(1) \oplus \Q(0)$. 
For our purposes we will need the twist $(\text{Sym}^n K_t)(1)$, which is the variation on the same connection where the respective filtrations are shifted as $F^{i}((\text{Sym}^n K_t)(1)) = F^{i+1}(\text{Sym}^n K_t)$ and $W_{i}(\text{Sym}^n K_t)(1)) = W_{i+2}(\text{Sym}^n K_t)$.

We use the splitting \eqref{split} to define a Hodge filtration on $\E_n$ as follows
\eq{Hodge-E-n}{
\bal 
&F^i\E_n=0\oplus s(F^i M_{\Delta^*}) ;\quad i\ge 0 \\
F^i\E_n &= F^i((\text{Sym}^nK_t)(1))\oplus M_{\Delta^*};\quad i<0.
\eal}

\begin{prop}\label{Griffiths-transversality} The Hodge filtration~\eqref{Hodge-E-n} on $\E_n$ satisfies Griffiths transversality: 
$$\nabla(F^a\E_n) \subset F^{a-1}\E_n\otimes \Omega^1_{\Delta^*}.
$$ 
\end{prop}

The following basic consequence of the maximal unipotency of the local monodromy of $M$ at $t=0$ will be useful in our proof of Proposition~\ref{Griffiths-transversality}:

\begin{lem}\label{mumhodge} Suppose $M$ is a connection on $U \subset \G_m$ carriying a polarized variation of pure HS and such that the monodromy of the local system of its flat sections around $t=0$ is maximally unipotent. Let $p\in \Delta^* \subset U$ and consider the pure Hodge structure on the fiber $M_p$. Then for some $a<b$ the Hodge graded  is given by $$gr_FM_p = \bigoplus_{a\le i\le b} F^iM_p/F^{i+1}M_p$$  with each $F^iM_p/F^{i+1}M_p \cong \C$.  
\end{lem}
\begin{proof}The point of Schmid's construction of the limiting mixed Hodge structure for $M$ at $0\in \Delta$ is that the Hodge filtration passes to a limit as $p \to 0$, so it suffices to prove the assertion for the Hodge graded of the limiting mixed Hodge structure $M_{lim}$. But $M_{lim}$ has a logarithm of monodromy operator $N: M_{lim} \to M_{lim}(-1)$. Let the variation of pure HS on $M$ be of weight $w$. Then the monodromy weight filtration $W_*M_{lim}$ is given by the shift $W_* = L_{*-w}$ of the Jacobson filtration $L_*M_{lim}$. In our case $N=\log(\sigma_0)$ has one Jordan block (see~\eqref{N-matrix}) and therefore the Jacobson filtration has a very simple form
\ml{Jacobson-on-Mlim}{M_{lim}=L_rM_{lim}=L_{r-1}M_{lim};\quad L_{r-2}M_{lim}=L_{r-3}M_{lim}=NM_{lim} \\
\cdots L_{-r+1}M_{lim} = N^{r-1}M_{lim};\quad L_{-r}M_{lim}=(0).
}
Finally, $N^iM_{lim}/N^{i+1}M_{lim}$ is a one dimensional vector space. For $i\ge 0$, the map $N:gr^i_NM \to gr^{i+1}_NM(-1)$ is a surjective map of Hodge structures of dimensions $\le 1$. For a Hodge structure $H$ we have $F^i(H(-1))=F^{i-1}H$, so there are two possibilities. Either $gr^{i+1}_NM=(0)$, or $gr^{i+1}_NM=\Q(k)$ for some $k$, in which case $gr^i_NM=\Q(k-1)$. Since $\C(k)$ has Hodge filtration concentrated in degree $-k$, the lemma follows. 
\end{proof}

As a consequence of Lemma~\ref{mumhodge}, the Hodge filtration on our connection $M$ is $0=F^rM \subset F^{r-1}M\subset F^{r-2}M\subset\cdots\subset F^0M=M$ with $F^i/F^{i+1}$ rank $1$. Since we assumed that $m \in M$ is chosen in the smallest part of the Hodge filtration, it follows that the weight is given by $w=r-1$ and 
\eq{Hodge-F-on-M}{
F^k M = \sum_{j=0}^{r-1-k} \sO_U D^j m.
}

\begin{proof}[Proof of Proposition~\ref{Griffiths-transversality}] 
Recall the $\sO$-splitting $s: M_{\Delta^*} \inj \E_n$ is defined by $D^im\mapsto D^i\phi^{an},0\le i\le r-1$. It follows from~\eqref{Hodge-F-on-M} that $F^i \E_n = \sum_{j=0}^{r-1-i} \sO D^j \phi^{an}$ for $i \ge 0$, which immediately implies the statement of the proposition for all $a>0$. To deal with $a \le 0$ we will need the following observation. 

\begin{lem}We have $D^jL\phi^{an} = e_{-1-j},\ 0\le j\le n.$
\end{lem}
\begin{proof}[Proof of lemma] Since $Lm=0$, we can write $L\phi^{an}=B_1e_{-1}+\cdots+B_{n+1}e_{-n-1}$ with $B_i\in \sO$. From \eqref{Lphi} we get
$$\bal&\frac{(\log t)^{k-r}}{(k-r)!} = L\phi_k = \langle L\phi^{an}, \ve_k\rangle \\
&= \langle B_1e_{-1}+\cdots + B_{1+n}e_{-n-1},\ve_k\rangle = \sum_{j=r}^{k} \frac{\log(t)^{k-j}}{(k-j)!} B_{j-r+1}.
\eal$$
Since the $B_i\in \sO$ cannot involve powers of $\log(t)$, we conclude that $B_1=1$ and $B_j=0$ for $j>1$. Thus, $L\phi^{an}= e_{-1}$ and hence $D^jL\phi^{an} =  e_{-1-j}$.  
\end{proof}

Note that $F^{-j}\Big((\text{Sym}^n K_t)(1)\Big) = \sO e_{-1}\oplus\cdots \oplus \sO e_{-j}$. Since $L\phi^{an} = e_{-1}$ we see that $D^r \phi^{an} \in F^{-1} \E_n$ (remember we shrank $\Delta^*$ so that the leading coefficient of $L$ is a unit in $\sO$.) The case $a=0$ follows from $D\left( F^0 \E_n\right) \subset \sum_{j=0}^r \sO D^j \phi^{an} = F^{-1}\E_n$. Finally, $(\text{Sym}^n K_t)(1)$ satisfies Griffiths transversality, and the proposition follows for all $a$. \end{proof}
 
It is clear from the proof of Proposition~\ref{Griffiths-transversality} that the filtration~\eqref{Hodge-E-n} is actually given by
\eq{Hodge-E-n-1}{
F^i \E_n = \sum_{j=0}^{r-1-i} \sO \, D^j \phi^{an}, \quad -n-1 \le k \le r-1.
}
With the Hodge filtration~\eqref{Hodge-E-n-1} and the weight filtration defined as 
\eq{weight-E-n}{W_i \E_n =W_i((\text{Sym}^nK_t)(1)) \;\text{  for }\; i<r-1;\quad 
W_{r-1}\E_n=\E_n,
}
it is clear that our variation of mixed Hodge structure on $\E_n$ is an extension of the pure  variation of weight $w =r-1$ on $M_{\Delta^*}$ by the mixed Tate variation $(\text{Sym}^n K_t)(1)$. 

As it stands, our variation of HS $\E_n$ is only a variation of $\C$-HS. We shall now lift it to a variation of $\Q$-HS. We work at a base point $p \in \Delta^*$. Recall that the $\Q$-structure on solutions $\sM_p^\vee \subset (\sE_n^\vee)_p$ is given by rational Betti homology classes.  The subspace $\sM_p^\vee(\Q)$ is clearly preserved by the monodromy. Since we assume that the action of $\sigma_0$ is maximally unipotent, we have the nilpotent operator $N=\log(\sigma_0)$ acting on $\sM_p$ and preserving the $\Q$-structure. Quite generally, associated to a nilpotent operator on a finite dimensional vector space there is a filtration $L_*$ (Jacobson filtration). In our case $N$ is maximally nilpotent and the Jacobson filtration is given by $L_k=N^{\lfloor \frac{r-k}{2}\rfloor}(\sM_p)$ for $1-r \le k \le r-1$ (this fact was already used in the proof of Lemma~\ref{mumhodge}, see~\eqref{Jacobson-on-Mlim}). As $M$ carries a pure HS of weight $w=r-1$, a basic result of Schmid and Deligne implies that there exists a limiting mixed HS $M_{\lim}$ with $M_{\lim}(\Q)=\sM_p(\Q)$ and weight filtration given by 
\[
W_k M_{\lim} = L_{k-w}= N^{r-\lceil \frac{k+1}2\rceil}(\sM_p),\quad 0 \le k \le 2(r-1).
\]
It is clear in our case that  $N^{r-1-j}(\sM_p^\vee(\C)) = \C \ve_0 + \ldots + \C \ve_j$. We shall show that, in the situation when the Frobenius constants $\kappa_j$ are defined  and $\kappa_0 \ne 0$, these numbers can be used to describe the $\Q$-subspaces of the Jacobson filtration. Recall that $t=c$ is a reflection point for $M$. Suppose that $\delta \in \sM_p^\vee$ in Definition~\ref{c-assum} is chosen to be rational, $\delta \in (\sigma_c-1)\sM_p^\vee(\Q)$, and consider  the Frobenius constants defined by
\eq{kappas-via-delta}{
(\sigma_c-1)\ve_j = \kappa_j \delta,
}
where $\ve_j$ are the horizontal sections~\eqref{horsol-eps} yielding the classical Frobenius solutions,~\eqref{Frob-basis-via-phi-an}. As a connection, we have a natural globalization of $\E_n$, namely $\E_n=\sD/\sD D^{n+1}L$, \eqref{connect}. Thus, it makes sense also to talk about the variation around $c$ for horizontal sections of $\E_n$. Lemma~\ref{c-remains-reflection} shows that the variation of the local monodromy of the local system $\sE_n^\vee = (\E_{n,an}^\vee)^{\nabla^\vee = 0} \cong Sol(D^{n+1}L)$ around $c$ also has rank $1$. Hence formula~\eqref{kappas-via-delta} with $j \ge r$ defines the higher Frobenius constants. Recall that solutions $\ve_0, \ldots,\ve_{n+r}$ form a $\C$-basis for the multivalued horizontal sections $(\sE_n^\vee)_p$; it will be convenient to drop the base point $p$ from our notation and denote their $\C$-span by $\sE^\vee_n = \C \ve_0 + \ldots + \C \ve_{n+r}$. 

\begin{prop}\label{Qstruct} Let $\kappa_0,\kappa_1,\ldots$  be the Frobenius constants defined by~\eqref{kappas-via-delta}, where $\delta \in \sM_p^\vee(\Q)$ is a non-zero rational solution spanning the image of the variation of the local monodromy around $t=c$. Assume further that $\sigma_c(\ve_0)\neq \ve_0$. Then there exists a unique $\Q$-structure $\sE^\vee_n(\Q)$ on $\sE^\vee_n$ such that \newline\noindent
(i) $\kappa_0^{-1} \ve_0\in \sE^\vee_n(\Q)$ and the filtration $\text{fil}_j \sE^\vee_n := \C \ve_0+\C\ve_1+\cdots+\C\ve_j$ is defined over $\Q$;\newline\noindent
(ii) $\sE^\vee_n(\Q)$ is stable under both loops $\sigma_0, \sigma_c$ around $t=0$ and $t=c$ respectively. \newline\noindent
A $\Q$-basis for  $\sE^\vee_n(\Q)$ is then given by 
\eq{eta}{
\eta_k := (2 \pi i)^{-k} \sum_{j=0}^k \alpha_j \ve_{k-j}, \quad 0 \le k \le n+r,
}
where $\alpha_j$ are the coefficients of the series
\eq{alpha}{
\sum_{j=0}^\infty \alpha_j s^j := 1 /\left(\sum_{j=0}^\infty \kappa_j s^j\right) = \frac1{\kappa_0} + \frac{-\kappa_1}{\kappa_0^2} \, s + \frac{-\kappa_2\kappa_0+\kappa_1^2 }{ \kappa_0^3} s^2+ \ldots}
\end{prop}

\begin{proof}The $\eta_k$ defined in~\eqref{eta} are linearly independent over $\C$ (because the $\ve_k$ are.) It is clear that $\sE_n^\vee(\Q) := \sum_{k=0}^{n+r} \Q \eta_k$ satisfies~(i). Let us check it also satisfies~(ii). One can easily check that $N = \log(\sigma_0)$ acts as $N \eta_k = \eta_{k-1}$, and hence $\sigma_0=\exp(N)$ preserves $\sE_n^\vee(\Q)$. As for the action of $\sigma_c$, we observe that
\[
(\sigma_c-1)\eta_k = (2 \pi i)^{-k}\sum_{j=0}^k \alpha_j (\sigma_c-1)\ve_{k-j} = (2 \pi i)^{-k} \left(\sum_{j=0}^k \alpha_j \kappa_{k-j} \right) \delta  = \begin{cases} 0, & k > 0, \\ \delta, & k=0. \end{cases} 
\]
It remains to show that $\delta \in \sE_n^\vee(\Q)$. We write $\delta = \sum_{k=0}^{r-1} \mu_k \eta_k$ with $\mu_k \in \C$. Since $\delta \in \sM_p^\vee(\Q)$ and this space is preserved by the monodromy operators, it follows from $(\sigma_c-1)\delta=\mu_0 \delta$ that $\mu_0 \in \Q$. To access other coefficients, note that $(\sigma_c-1)N^k \delta = \mu_k \delta$ implies that $\mu_k \in \Q$ for all $k$.

To show uniqueness, denote $V_k = \langle\eta_0,\dotsc,\eta_k\rangle_\Q$ and suppose we have another $\Q$-structure with these properties, say $W_k = \langle\nu_0,\dotsc,\nu_k\rangle_\Q$ with $W_k \otimes \C = V_k \otimes \C$ for $k \ge 0$. It is clear from~(i) that $V_0=W_0$. Suppose $k>0$ and assume inductively that $W_{k-1}=V_{k-1}$. Due to~(ii) we must have that $\delta=(\sigma_c-1)\eta_0 \in W_{r+n}$, and hence for any $w \in W_k$ one has $(\sigma_c-1) w\in \Q \delta$. It follows that $W_k = W_0\oplus W_k^0$ where $W_k^0:=\ker(\sigma_c-1: W_k \to W_{r+n})$.  Since $(\sigma_0-1)(W_0)=(0)$, it follows that $W_k^0\subset W_k$ defines a splitting of $\sigma_0-1: W_k \surj W_{k-1}=V_{k-1}$. We have $V_k^0\otimes\C = W_k^0\otimes\C \subset \langle\ve_0,\dotsc,\ve_k\rangle$ and with this identification, the two splittings coincide. Thus $V_k^0=W_k^0$ so $V_k=W_k$. 
\end{proof}

Note that $fil_{r-1} \sE_n^\vee(\Q)=\sM_p^\vee(\Q)$. Indeed, $\delta=(\sigma_c-1)(\kappa_0^{-1} \ve_0) \in  \sE_n^\vee(\Q)$ and under the condition $\sigma_c \ve_0 \ne \ve_0$ the space of Betti cycles $\sM_p^\vee(\Q)$ is generated by the images of $\delta$ under $\sigma_0$ (see the proof of  Corollary~\ref{MUM-case-of-main-thm}). Since $fil_j \sE_n^\vee = N^{n+r-j}(\sE_n^\vee)$, Proposition~\ref{Qstruct} defines a unique lift of the Betti structure on $\sM^\vee$ to a $\Q$-structure on $\sE_n^\vee$ for which the Jacobson filtration for $N=\log(\sigma_0)$ is defined over $\Q$.

The final objective in this section is to link the Frobenius constants $\kappa_j$ to periods of a limiting mixed Hodge structure. As we mentioned earlier, for a pure variation of Hodge structure $M$ of weight $w$ the limiting mixed Hodge structure $M_{lim}$ can be identified with the fiber $\sM_p$ with the weight filtration given by the Jacobson filtration for $N=\log(\sigma_0)$ shifted by $w$. More generally, for a variation of mixed Hodge structure $H$ with the weight filtration $W_*H$, a \emph{monodromy weight filtration} on the fiber $\sH_p$ is a filtration $\sW_* \sH_p$ such that $N(\sW_j) \subset \sW_{j-2}$ and such that for each $k$ the filtration induced by $\sW_*$ on $gr_k^W \sH_p$ is the Jacobson filtration defined by $N$ on the pure HS $gr_k^W \sH_p$ and then shifted by $-k$. There is at most one monodromy weight filtration satisfying these conditions, but it can happen that no such monodromy weight filtration exists. As earlier, we denote the fiber $(\sE_n)_p$ simply by $\sE_n$.  

\begin{prop}\label{monodromy-weight-filtration-E-n} Let $L_*$ be the Jacobson filtration for $N=\log(\sigma_0)$ on $\sE_n$. Then $\sW_* = L_*[n+2-r]$ is the monodromy weight filtration on $\sE_n$.
\end{prop}
\begin{proof} One can check that $N^k(\sE_n) = (\C \ve_0^\vee + \ldots + \C \ve_{k-1}^\vee)^\perp = \C e_{-n-1}+\ldots+\C e_{r-1-k}$, which yields 
\[\bal
\sW_k \sE_n &= L_{k+n+2-r} \sE_n = N^{\lfloor \frac{n+r+1-(k+n+2-r)}{2}\rfloor}(\sE_n)= N^{r-1-\lfloor \frac{k}{2}\rfloor}(\sE_n)\\
& = \C e_{-n-1} + \ldots + \C e_{\lfloor \frac{k}{2}\rfloor}.
\eal\]
The weight filtration $W_* \E_n$ is given by~\eqref{weight-E-n}. One can check that for $k < 0$ we have 
$W_{k}\sE_n = \sW_k \sE_n$ and hence the filtration induced by $\sW_*$ on each $gr^W_{-2k}\sE_n$ is zero in degrees $< -2k$ and everything in degrees $\ge -2k$. This is precisely the Jacobson filtration for $N$ on this rank one subquotient shifted by $2k$. It remains to check the rank $r$ graded piece $gr_{r-1}^W \sE_n \cong \sM$. There the induced filtration is given by 
\[
\sW_k gr_{r-1}^W \sE_n = \C \overline {e}_0 + \ldots + \C \overline {e}_{\lfloor \frac{k}{2}\rfloor}  = N^{r-1-\lfloor \frac{k}{2}\rfloor}\left( gr_{r-1}^W \sE_n \right),     
\]
which is again the Jacobson filtration shifted by $1-r$.
\end{proof}

\begin{prop}\label{alphas-are-periods} The limiting mixed Hodge structure $(\E_n)_{lim}$ exists. If $\alpha_k$ are the coefficients of the inverse Frobenius series~\eqref{alpha}, then numbers $\alpha_k (2 \pi i)^{-h}$ with $0 \le k \le h \le n+r$ are periods of $(\E_n)_{lim}$.  
\end{prop}

The following observation will be a key step in the proof of Proposition~\ref{alphas-are-periods}. We shall apply the limiting process (as $t \to 0$, a MUM point) from~\cite[(6.15)]{WS} to linear functionals that are naturally defined on the space of solutions $V=Sol_p(L)$ of a differential operator $L$ near a base point $p$. 
Namely, such linear functionals consist of application to elements of $V$ of various differential operators $\sum_j v_j(t) D^j \in \sO[D]$ followed by analytic continuation along a path and evaluation at its endpoint; as the endpoint approaches $t=0$ the monodromy is an obstruction to taking the limit. Schmid's limiting process involves killing the monodromy. When $t=0$ is a MUM point, it can be symbolically expressed by the formula
\[
\lim_{t \to 0} \; \sigma_0^{-\frac{\log(t)}{2 \pi i}} \sum_{j=0}^{m-1} v_j(t) D^j = \sum_{j=0}^{m-1} v_j(0) \phi_j^\vee \in V^\vee, 
\]  
where $\phi_0^\vee,\ldots,\phi_{m-1}^\vee$ is the basis in $V^\vee = {\rm Hom}_\C(V,\C)$ dual to the classical Frobenius basis in $V$ (see Lemma~\ref{key-lemma-limiting-MHS}). This observation allows us to define a mixed Hodge structure on the space of solutions which for Picard--Fuchs differential operators, by construction, coincides with the limiting MHS at $t=0$ defined by Schmid in~\cite{WS}. As a consequence, the coefficients of Betti-rational solutions expressed in the Frobenius basis at $t=0$ are periods of the respective limiting MHS (see Remark~\ref{PF-application-remark}). Note that both the limiting MHS and the set of Frobenius solutions depend on the choice of a branch of $\log(t)$.  

\begin{lem}\label{key-lemma-limiting-MHS} Let $K \subset \C$ be a field and $L = \sum_{j=0}^m q_j(t) D^{m-j}\in K(t)[D]$ be a differential operator with rational coefficients such that the monodromy of its solutions around $t=0$ is maximally unipotent. Assume that the unqiue local exponent of $L$ at $t=0$ equals $\rho=0$. Let $V = Sol_p(L)$ be the space of solutions near a base point $t=p$.   Fix some branch of $\log(t)$ near $p$ and consider the classical Frobenius basis in $V$:
\[
\phi_k(t) = \sum_{j=0}^k \frac{\log(t)^j}{j!} \phi_{k-j}^{an}(t), \qquad k=0,\ldots,m-1,
\]
where $\phi_0^{an},\ldots, \phi_{m-1}^{an} \in K\llbracket t \rrbracket$ are analytic near $t=0$ functions satisfying $\phi_j^{an}(0)=\delta_{j,0}$.

(i) Let $\sO$ be the ring of analytic functions on the punctured disc $\Delta^* = \Delta \setminus \{0\} = \{ t \in \C : 0 < |t| < \ve \}$ that are meromorphic at $t=0$. Denote $\sD = \sO[D]$ and consider the analytic connection on $\Delta^*$ given by $H = \sD / \sD L \cong \sum_{j=0}^{m-1} \sO D^j$. Consider a filtration on $H$ given by 
\eq{F-k-on-H-in-lemma}{
F^k H = \sum_{j=0}^{m-1-k} \sO D^j, \quad 0 \le k \le m-1.
}
Then the limiting filtration in the sense of Schmid (\cite[(6.15)]{WS})  exists and is given on the dual space $V^\vee$ by
\[
F^k_{\infty} \; V^\vee = \C \phi_0^\vee + \ldots + \C \phi_{m-1-k}^\vee.
\] 
Moreover, limits of algebraic classes span $F^k_\infty(K) :=  K \phi_0^\vee + \ldots + K \phi_{m-1-k}^\vee$.

(ii) Consider the nilpotent operator $N = \log(\sigma_0)$ and define $W_*$ to be the respective Jacobson filtration on $V^\vee$ shifted by $m-1$; this filtration is given by 
\eq{W-k-on-H-in-lemma}{
W_k V^\vee := N^{m-1-\lfloor \frac k 2\rfloor}(V^\vee).
}
Then $\H=(V^\vee, W_*, F^*_\infty)$ is a mixed Hodge structure with $gr^{\sW} \H = \oplus_{k=0}^{m-1} \C(-k)$. 
\end{lem} 
\begin{proof}  (i) Consider the universal covering \[
e: G \to \Delta^*, \quad e(z) = \exp(2 \pi i z) 
\]
of $\Delta^*$ by the upper halfplane $G = \{ z \in \C | Im(z) > -\frac1{2 \pi} \log(\ve) \}$. We will  identify the space of solutions with
\[
V = \left\{ u(z) \in \sO_G \;|\;  (e^*L) u = 0   \right\}.
\] 
The monodromy transformation $\sigma_0 : V \to V$ acts by 
\eq{M0}{
(\sigma_0 u)(z)=u(z+1).
}
For each $z \in G$ there is a map $\pi_z: H \to V^\vee$. In concrete terms, for an element $v= \sum_{j=0}^{m-1} v_j(t) D^j \in H$ the pairing of $\pi_z(v)$ with a solution $u \in V$ is given by 
\[
\langle \pi_z(v) , u \rangle = \sum_j v_j(e(z)) (e^* D)^j u = \sum_j \frac{v_j(e(z))}{(2 \pi i)^j} \frac{d^j u}{dz^j}(z). 
\] 
(Here we used $e^* D = e(z) \frac{d}{d e(z)} = \frac1{2 \pi i} \frac d{dz}$.) Let $N = \log(\sigma_0)=\sum_{h \ge 1} (-1)^{h-1} (\sigma_0-I)^h/h$ be the logarithm  of the local monodromy transformation. To apply Schmid's limiting procedure as in \cite[\S 6]{WS}, we shall consider maps defined by 
\[
\pi'_{z} := \exp(- z N) \circ \pi_{z} : H \to V^\vee.
\]
To prove~(i) we will check that for any $v = \sum_{j=0}^{m-1} v_j(t) D^j \in H$:
\begin{itemize}
\item[(i')]  one has $\pi'_{z+1}(v) = \pi'_{z}(v)$;
\item[(i'')] the limit $\lim_{Im(z) \to +\infty} \pi_z'(v)$ exists if and only if all $v_j$ are analytic at $t=0$, in which case one has
\eq{Schmid-limit}{
\lim_{Im(z) \to +\infty} \pi'_{z}(v) =  \sum_{j=0}^{r-1} v_j(0) \phi_j^\vee, 
}
where $\phi_0^\vee, \phi_1^\vee, \ldots$ is the basis in $V^\vee$ dual to the Frobeius basis
\eq{Frob-basis-univ-cover}{
\phi_k(z) = \sum_{j=0}^k \frac{(2 \pi i z)^j}{j!} \phi_{k-j}^{an}(e(z)) \in V, \quad 0 \le k \le m-1.
}
\end{itemize}
Note that every solution $u \in V$ can be uniquely written as $u(z) = \sum_{k=0}^{m-1} u_k(e(z)) z^k$ 
with $u_0(t),\ldots,u_{m-1}(t) \in \sO_\Delta$. Using~\eqref{M0} we find that the action of $N$ on $V$ is given by
\[
N : \sum_k u_k(e(z)) z^k \mapsto \sum_k u_k(e(z)) k \, z^{k-1}.
\]
We evaluate $\pi_z'(v)$ on $u = \sum_k u_k(e(z)) z^k$ as follows:
\[\bal
\langle \pi'_{z}(v), u\rangle &= \sum_{h \ge 0} \frac{(-z)^h}{h!} \langle \pi_z(v), N^h u \rangle \\
&= \sum_{h \ge 0} \sum_{j,k=0}^{m-1} \frac{(-z)^h}{h!} \frac{v_j(e(z))}{(2 \pi i)^{j}} \left(\frac{d}{dz}\right)^j \left( u_k(e(z))   \left(\frac{d}{dz}\right)^h z^k \right) \\
&= \sum_{h \ge 0} \sum_{j,k=0}^{m-1} \frac{(-z)^h}{h!} v_j(e(z)) \sum_{s=0}^j (D^{j-s} u_k)(e(z))  \frac{k(k-1)\ldots(k-h-s+1)}{(2 \pi i)^{s}} z^{k-h-s} \\
&= \sum_{j,k=0}^{m-1} \sum_{s=0}^{\min(j,k)} v_j(e(z)) (D^{j-s} u_k)(e(z))(2 \pi i)^{-s} \frac{k!}{(k-s)!} z^{k-s}\sum_{h \ge 0} (-1)^h \binom{k-s}{h} \\
&= \sum_{j \ge k} v_j(e(z)) (D^{j-k} u_k)(e(z)) \frac{k!}{(2 \pi i)^{k}}.\\
\eal\]
Note that the function in the last row is periodic in $z$, which proves the claim (i'). When all $v_j(t)$ are analytic at $t=0$ the above expression passes to the limit as $Im(z)$ grows infinitely:
\[
\lim_{Im(z) \to +\infty} \langle \pi'_{z}(v) , u \rangle =  \sum_{j=0}^{m-1} v_j(0) u_j(0) \frac{j!}{(2 \pi i)^j}.
\] 
It remains to notice that the linear functional that sends every $u=\sum_k u_k(e(z)) z^k$ to $u_j(0) (2 \pi i)^{-j}\,j!$ coincides with $\phi_j^\vee$. This fact follows from formula~\eqref{Frob-basis-univ-cover} because $\phi_0^{an}(0)=1$ and $\phi_j^{an}(0)=0$ when $j>0$. This completes our proof of~(i''). Algebraic classes $v = \sum_{j=0}^{m-1} v_j(t) D^j \in H$ are those with $v_j \in K(t)$. All claims of part~(i) now follow from formula~\eqref{Schmid-limit}.

For~(ii) we note that $N^a(V^\vee)= Span_\C(\phi_a^\vee,\ldots,\phi_{m-1}^\vee)$ and hence filtrations $F^*_\infty H$ and $W_*$ are opposite in the sense that $V^\vee = W_{2k} \oplus F^{k+1}_{\infty}$ for any $k$. It follows that the filtration induced by $F^*_{\infty}$ on $gr_{2k}^W \, V^\vee$ is zero in degrees $> k$ and everything in degree $k$. 
\end{proof}

\begin{remark}\label{PF-application-remark} Suppose that $L$ in Lemma~\ref{key-lemma-limiting-MHS} is a Picard--Fuchs differential operator. More precisely, we assume that for a smooth projective family of algebraic varieties $f: X \to U$ there is a class in the smallest Hodge part $\omega \in F^{m-1} H_{dR}^{m-1}(X/U)$ annihilated by the differential operator $L$. Then $H=\sD/\sD L$ carries a polarized variation of Hodge structure of pure weight $m-1$, and using Griffiths' transversality along with Lemma~\ref{mumhodge} we conclude that~\eqref{F-k-on-H-in-lemma} is the Hodge filtration. The limiting MHS is constructed in \cite[Theorem (6.16)]{WS}, and this is precisely $\H$ from (ii) in Lemma~\ref{key-lemma-limiting-MHS}. 

The $K$-structure from part~(i) of our Lemma is the de Rham structure on $\H$. To see this, we change notation in order to appeal to the work of Steenbrink (\cite{JS}, as corrected in \cite{I}). Steenbrink considers a geometric situation where $H$ is the DR-cohomology of a projective family $f:\sX\to S$, where $\sX$ is smooth and $S$ is a smooth, affine curve. $t$ is a parameter on $S$ and $f$ is smooth away from $t=0$. We assume $Y:=f^{-1}(0)$ is a reduced normal crossings divisor. The link with our standard notation $f:X \to U$ is $X=\sX-Y;\ \  U=S-\{t=0\}$.

Define (here $\sO_{S,0}$ is the local ring at $0\in S$ and $K$ is the residue field at $0$.) 
\[\omega^*_S = \Omega^*_S(\log \{0\});\quad \omega^*_\sX = \Omega^*_\sX(\log Y);\quad \omega^*_{\sX/S} = \Omega^*_{\sX/S}(\log Y);\quad \omega^*_Y = \omega^*_{\sX/S}\otimes_{\sO_{S,0}} K.
\]

Steenbrink's basic result identifies $\omega^*_Y\otimes_K \C$ with the nearby cycle complex $R\Psi(\C)$ for $Y\subset \sX$. This identification depends on the choice of $t$ and also of $\log t$. Given $n$, Steenbrink's result enables one to put a mixed Hodge structure on $H^n(Y,\omega^*_Y)$ which is then identified with the limiting MHS $H^n_{lim}$ as defined by Deligne and Schmid. The nearby cycle complex carries a Betti $\Q$-structure which gives a Betti $\Q$-structure on $H^n_{lim}$, (\cite{PS}, Chap. 11.2). The fact that $\omega^*_Y$ is defined algebro-geometrically automatically endows $H^n_{lim}$ with a $K$-structure (DR structure) which can be used to define periods. 

We introduce a variable denoted $\log t$ and consider the complexes
$$\omega_S[\log t];\quad  \omega_\sX[\log t];\quad \omega_{\sX/S}[\log t].
$$
Sections e.g. of $\omega_S^*[\log t]$ are polynomials in $\log t$ with coefficients which are sections of $\omega^*$. The differential is extended from $\omega$  by setting $d\log t=dt/t$. Note that $dt/t=0$ in $\omega_{\sX/S}$. 

Let $i:Y \inj \sX$ be the inclusion, and write $i^{-1}$ for the sheaf-theoretic restriction functor from sheaves on $\sX$ to sheaves on $Y$. (Note $i^{-1}\neq i^*$, the pullback in the category of sheaves of $\sO$-modules.) Steenbrink's basic result is that the composition
\eq{st-2}{i^{-1}(\omega_\sX[\log t]) \xrightarrow{\log t\mapsto \log t} i^{-1}(\omega_{\sX/S}[\log t]) \xrightarrow{\log t\mapsto 0} i^{-1}\omega_{\sX/S} \to \omega_Y
}
is a quasi-isomorphism. 

Consider the diagram \minCDarrowwidth.5cm
\eq{st-3}{\begin{CD}0 @>>> i^{-1}(\omega^{\cdot -1}_{\sX/S} \otimes \omega^1_S[\log t]) @>>> i^{-1}(\omega_\sX[\log t]) @>\alpha >> i^{-1}(\omega_{\sX/S}[\log t]) @>>> 0 \\
@. @VVV @| @VVV \\
0 @>>> \ker \beta @>>> i^{-1}(\omega_\sX[\log t]) @>\beta >> i^{-1}(\omega_{\sX/S}) @>>> 0 \\
@. @AAA @AAA @| \\
0 @>>> i^{-1}(\omega^{\cdot -1}_{\sX/S} \otimes \omega^1_S) @>>> i^{-1}(\omega^*_\sX) @>>> i^{-1}(\omega^*_{\sX/S})@>>> 0. \\
@. @. @. @VVV @. \\
@. @. @. \omega_Y
\end{CD}
}

A piece of the long-exact sequence of cohomology sheaves on $Y$ associated to the top line reads (for any $p$)
\ml{st-4}{0 \to \sH^p(i^{-1}(\omega_\sX[\log t])) \xrightarrow{\sH^p(\alpha)} \sH^p( i^{-1}(\omega_{\sX/S}[\log t])) \\
 \xrightarrow{\nabla_{GM}} \sH^p(i^{-1}(\omega^{\cdot}_{\sX/S} \otimes \omega^1_S[\log t]))
}
The boundary map coincides with the Gau\ss--Manin connection as indicated. Also, the result of Steenbrink cited above implies that $\sH^p(\alpha)$ is injective, both on the sheaf level and for global cohomology groups. Thus, \eqref{st-4} identifies
\eq{st-5}{H^p(Y, i^{-1}(\omega^*_\sX[\log t]) \cong H^p(Y, i^{-1}(\omega_{\sX/S}[\log t]))^{\nabla_{GM}=0}
}
We know by Frobenius that we have a full set of horizontal sections defined over $K \llbracket t \rrbracket [\log t]$, so we conclude
\ml{st-6}{H^p(Y, i^{-1}(\omega^*_\sX[\log t]) \cong \\
\{\text{Horizontal sections of the GM connection on $H^p(Y,i^{-1}\omega^*_{\sX/S})$}\} \\
\cong H^p(Y,\omega^*_Y).
}
The assignment $\log t \to 0$ in \eqref{st-2} coincides with the vanishing of $z^j, j>0$ in the computation of~\eqref{Schmid-limit}. The $K$-structure from $H^p(Y,\omega_Y)$ is Steenbrink's $DR$-structure. It matches the $K$-structure on $\H$, and if one expresses period functions in the classical Frobenius basis, the coefficients are periods of the limiting Hodge structure.     
\end{remark}

\begin{ex} The period function of the Legendre family of elliptic curves
\[
\phi(t) = \int_1^\infty \frac{dx}{\sqrt{x(x-1)(x-t)}} = \pi \sum_{n=0}^\infty \binom{2n}{n}^2 \left( \frac t{16}\right)^n = \pi \cdot \, _2F_1(\tfrac12,\tfrac12,1|t)
\] 
is a Betti-rational solution to the hypergeometric differential operator $L = D^2 - t(D+\tfrac12)^2$. Then $\pi$ here is the period of the limiting MHS because the hypergeometric function $\phi_0(t)= \, _2F_1(\tfrac12,\tfrac12,1|t)$ is the Frobenius solution analytic at $t=0$. 
\end{ex}

Lemma~\ref{key-lemma-limiting-MHS} can be also applied to variations of mixed Hodge structure. Namely, if the Hodge filtration is a shift of $F^*H$ and if the same shift of $\sW_*$ yields the monodromy weight filtration, then the respective shift of $\H$ is the limiting mixed Hodge structure.

\begin{ex} $L=D^{n+1}$ corresponds to the symmetric power of the Kummer variation $Sym^n K_t$. One can check that the Hodge filtration is given by the shift $F^*[n]=F^{*+n}$ of the filtration~\eqref{F-k-on-H-in-lemma} and that the shift $W_*[n]=W_{*+2n}$ of the filtration~\eqref{W-k-on-H-in-lemma}  is the monodromy weight filtration. It follows that the limiting Hodge structure is $(Sym^n K_t)_{lim}=\H[n]$. The Frobenius basis is given by $\phi_k(t)=\frac{\log(t)^k}{k!}$ and the rational structure is spanned by $\left(\frac{\log(t)}{2 \pi i}\right)^k = (2 \pi i)^{-k} k! \, \phi_k(t)$. Hence $(2 \pi i)^{-k}$, $0 \le k \le n$ are periods of $(Sym^n K_t)_{lim}$. 
\end{ex}    
    
\begin{proof}[Proof of Proposition~\ref{alphas-are-periods}.] We apply Lemma~\ref{key-lemma-limiting-MHS} for the operator $D^{n+1}L$ where $L$ is an operator of order $r$ satisfying the usual assumptions used throughout this section. Here $m=n+r+1$ and as connections we have $H=\sD/\sD D^{n-1} L \cong \E_n$, see~\eqref{connect}. By~\eqref{Hodge-E-n-1} the Hodge filtration $F^* \E_n$ is the shift of $F^* H$ in~\eqref{F-k-on-H-in-lemma} by $n+1$. By Proposition~\ref{monodromy-weight-filtration-E-n} the monodromy weight filtration $\sW_* \sE_n$ is the shift of $W_* H$ by $n+1$. We conclude that the limiting MHS for our variation $\E_n$ exists and is given by $(\E_n)_{lim}=\H[n+1]$. 

The rational structure on $\sE_n^\vee$ was defined in Proposition~\ref{Qstruct}. Since $\eta_k \in \sE_n^\vee(\Q)$ and
\[
\langle \phi^{an},\eta_k \rangle = (2 \pi i)^{-k} (\alpha_k \phi_0 + \alpha_{k-1} \phi_1 + \ldots + \alpha_0 \phi_k),
\]
then $\alpha_k (2 \pi i)^{-h}$ with $0 \le k \le h \le n+r$ are periods of $(\E_n)_{lim}$. 
\end{proof}

We showed that numbers $\alpha_k$ from Proposition~\ref{Qstruct} divided by certain powers of $2 \pi i$ are periods of the limiting MHS associated to the extension 
\eq{exten}{0 \to \text{Sym}^n(K_t)(1) \to \E_n \to M_{\Delta^*} \to 0.
}
Though we assume $M$ is motivic, i.e. $M$ is the Gau\ss--Manin connection for a family of varieties over $\P^1$ as in the beginning of this section, it is not clear that the extension \eqref{exten} is motivic. Indeed, we do not expect it to be so in general. To better understand this question, we consider briefly some calculations inspired by work of Kerr \cite[Section 5.3]{K}. Kerr considers the example of Frobenius which is a pencil of K3-surfaces defined by $1-t \, f(x_1,x_2,x_3)=0$ with 
$$f = \frac{(x_1-1)(x_2-1)(x_3-1)(1-x_1-x_2+x_1x_2-x_1x_2x_3)}{x_1x_2x_3},
$$
as in our Example~\ref{apery-example}.
He shows that the Milnor symbol $\{x_1,x_2,x_3\}$ defines classes in motivic cohomology $H^3_{mot}(X_\lambda,\Q(3))$ where $\lambda=1/t$ and $X_\lambda$ is a suitable compactification of the divisor $f(x_1,x_2,x_3)=\lambda$ in $\G_m^3$. Associated to such a motivic class, one has the Beilinson regulator 
$$reg(\{x_1,x_2,x_3\})\in \text{Ext}^1_{MHS}(H^2(X_\lambda, \Q(0)), \Q(1)). 
$$
We speculate that, writing $t=1/\lambda$ in \eqref{exten} as $\lambda \to \infty$, the extension $reg\{x_1,x_2,x_3\}$ coincides with \eqref{exten} for $n=0$.  If so, this will say in particular that \eqref{exten} for $n=0$ is in this case motivic. (NB. ``speculation''$<<$ ``conjecture'' $<<$ ``theorem'')

We may try to go further and consider the Frobenius example for $n=1$. The extension \eqref{exten} becomes
$$0 \to K_{\lambda^{-1}}(1) \to \E_1 \to H^2(X_\lambda,\Q) \to 0
$$
This extension lies in 
\ml{}{\text{Ext}^1_{MHS}(\Q(0), H^2(X_\lambda)\otimes K_{\lambda^{-1}}(3))\cong \\
\text{Ext}^1_{MHS}(\Q(0), H^3(X_\lambda\times (\G_m,\{1,\lambda^{-1}\}),\Q(4)))
}
Formally, we would expect such a class to arise as the Beilinson regulator of a relative motivic class in $H^4_{mot}(X_\lambda\times (\G_m,\{1,\lambda^{-1}\}),\Q(4))$. 

Actually, it is more precise to look at the whole family, allowing $\lambda$ to vary. To this end, consider the pair
$$\Big(\G_m\times \G_m,(\G_m\times \{1\})\cup \Delta_{\G_m}\Big)
$$
where $\Delta_{\G_m}$ is the diagonal. We view this as a family over $\G_m$ via $pr_1:\G_m\times\G_m \to \G_m$, and we want a class in 
$$H^4_{mot}(X\times_{\G_m}(\G_m\times \G_m,(\G_m\times\{1\})\cup \Delta_{\G_m}),\Q(4))
$$
Let $u$ be the coordinate in the righthand $\G_m$ factor. We consider the Milnor symbol $\{x_1,x_2,x_3,u\}$. Informally speaking, to define a relative motivic class, we need to trivialize this symbol along the diagonal $u=\lambda=f(x_1,x_2,x_3)$. (A convenient and rigorous treatment of relative motivic cohomology can be given using higher cycle complexes, but here our intention is merely to suggest a way forward. We do not attempt to give details.) Informally, one trivializes this symbol by invoking the Steinberg relations, viz.
\ml{}{\{x_1,x_2,x_3,u\} = \{x_1,x_2,x_3,f(x_1,x_2,x_3)\} = \\
\{x_1,x_2,x_3,\frac{(x_1-1)(x_2-1)(x_3-1)(1-x_1-x_2+x_1x_2-x_1x_2x_3)}{x_1x_2x_3}\} \\
= \{x_1,x_2,x_3,1-x_1-x_2+x_1x_2-x_1x_2x_3\} = \\
\{x_1,x_2,x_3,(1-x_1)(1-x_2)\Big(1-\frac{x_1x_2x_3}{(1-x_1)(1-x_2)}\Big)\} = \\
\{x_1,x_2,x_3,\Big(1-\frac{x_1x_2x_3}{(1-x_1)(1-x_2)}\Big)\} = \\
\{x_1,x_2,\Big(\frac{(1-x_1)(1-x_2)}{x_1x_2}\Big),
\Big(1-\frac{x_1x_2x_3}{(1-x_1)(1-x_2)}\Big)\} = 1
}
(On the last line we use the identity $\{x,1-ax\}=\{a^{-1},1-ax\}$.) 

Intuitively, at least, the above argument can be used in the Frobenius example to construct our extension motivically for $n=1$. We may hope to apply a similar argument for $n>1$, working with
$$X\times_{\G_m} \sG^n; \sG= (\G_m\times \G_m, \G_m\times 1\cup \Delta_{\G_m}).
$$
Here $\sG^n = \sG\times_{\G_m}\cdots\times_{\G_m} \sG$ where the structure maps are again $pr_1:\G_m\times \G_m \to \G_m$. The symbol becomes $\{x_1,x_2,x_3,u_1,\dotsc,u_n\}$. The first order trivializations along diagonals are as above, but there are now higher order compatibilities on multiple diagonals that are not understood.


\begin{thebibliography}{x}
\bibitem{Apery79} R. Ap\'ery, \emph{Irrationalit\'e de $\zeta(2)$ et $\zeta(3)$}, Ast\'erisque 61 (1979), 11--13

\bibitem{Beauville} A. Beauville, \emph{Les familles stables de courbes elliptiques sur $\P^1$ admettant quatre fibres singuli\`eres}, C. R. Acad. Sc. Paris 294 (1982), 657--660

\bibitem{BE} S. Bloch, H. Esnault, \emph{Homology for irregular connections}, J. Th\'eor. Nombres Bordeaux 16 (2004), no. 2, 357--371

\bibitem{B-H} F. Beukers, G. Heckman, \emph{Monodromy for the hypergeometric function $_nF_{n-1}$}, Invent. math. 95, 325--354 (1989)

\bibitem{Broadhurst} D. Broadhurst, \emph{private communication}

\bibitem{FJ} J. Fres\'an, P. Jossen,  \emph{Exponential Motives} (to appear)
\bibitem{G} S. Galkin, \emph{The Conifold Point}, ArXiv 1404.7388 Math.AG
\bibitem{GZ} V. Golyshev, D. Zagier, \emph{Proof of the Gamma Conjecture for Fano 3-folds of Picard rank one}, Izvestiya: Mathematics 80:1, 24--49

\bibitem{Hain} R. M. Hain, \emph{Classical polylogarithms}, Motives ({S}eattle, WA, 1991), Proc. Sympos. Pure Math. 55, 3--42 

\bibitem{Haefliger} A. Haefliger,~\emph{Local theory of meromorphic connections in dimension~$1$ (Fuchs theory)}, Chapter III in A. Borel et al. \emph{Algebraic D-modules}, 129--150 

\bibitem{I} L. Illusie, \emph{Autour du th\'eor\`eme de monodromie locale}, Expos\'e I, P\'eriodes $p$-adiques - S\'eminaire de Bures, 1988

\bibitem{Katz70} N. Katz, \emph{Nilpotent connections and the monodromy theorem :
applications of a result of Turrittin}, Publications math\'ematiques de l’I.H.\'{E}.S. 39 (1970),  175--232

\bibitem{Katz87} N. Katz, \emph{On the calculation of some differential Galois groups}, Invent. math. 87, 13--61

\bibitem{K} M. Kerr, \emph{Motivic Irrationality Proofs}, arXiv:1708.03836.
\bibitem{K2} M. Kerr, \emph{Unipotent extensions and differential equations (after Bloch-Vlasenko)}, to appear.
\bibitem{KZ-Periods} M. Kontsevich, D. Zagier, \emph{Periods}, in \emph{Mathematics unlimited---2001 and beyond}, Springer, 2001, 771--808.
  .

\bibitem{LS} F. Loeser, C. Sabbah, \emph{Equations aux diff\'erences finies et d\'eterminants d'int\'egrales de fonctions multiformes}, Comment. Math. Helvetici 66 (1991) 458--503
\bibitem{M} B. Malgrange, \emph{\'Equations Diff\'erentielles \`a Coefficients Polynomiaux}, Birkh\"auser, Progress in Math, vol. 96, (1991)
\bibitem{PS} C. Peters, J. Steenbrink, \emph{Mixed Hodge Structures}, Ergebnisse Mathematik, Springer Verlag (2007)
\bibitem{WS} W. Schmid, \emph{Variation of Hodge Structure: the Singularities of the Period Mapping}, Inventiones Math. 22, 211--319 (1973)
\bibitem{JS} J. Steenbrink, \emph{Limits of Hodge Structures}, Inventiones Math. 31, (1976), 229--257

\bibitem{Z-Apery} D. Zagier, \emph{Integral solutions of {A}p\'{e}ry-like recurrence equations}, in {\em Groups and symmetries}, volume~47 of {\em CRM Proc. Lecture Notes}, 349--366, Amer. Math. Soc., Providence, RI, 2009.
\end{thebibliography}
\end{document}